\documentclass[a4paper]{amsart}
\usepackage[margin=1in,headheight=13.6pt]{geometry}
\usepackage[T1]{fontenc}
\usepackage[utf8]{inputenc}
\usepackage{lmodern}
\usepackage[leqno]{amsmath}
\usepackage{amssymb,amsthm,mathtools,mathrsfs}
\usepackage{array,booktabs,longtable}
\usepackage{microtype,needspace}
\usepackage{enumitem}
\usepackage[hidelinks]{hyperref}
\usepackage{url}
\hypersetup{
  pdftitle={An elliptic Calabi-Yau threefold of Mordell-Weil rank twelve},
  pdfauthor={Dominik Burek},
  pdfsubject={Elliptic Calabi-Yau threefolds and Mordell-Weil groups},
  pdfkeywords={Calabi-Yau threefold, Mordell-Weil rank, crepant resolution, K3 surface}
}

\newtheorem{Thm}{Theorem}[section]
\newtheorem{Lem}[Thm]{Lemma}
\newtheorem{Pro}[Thm]{Proposition}
\newtheorem{Cor}[Thm]{Corollary}
\theoremstyle{definition}

\theoremstyle{remark}
\newtheorem{Rem}[Thm]{Remark}
\numberwithin{equation}{section}

\newcommand{\CC}{\mathbb C}
\newcommand{\QQ}{\mathbb Q}
\newcommand{\ZZ}{\mathbb Z}
\newcommand{\PP}{\mathbb P}
\newcommand{\FF}{\mathbb F}
\newcommand{\OO}{\mathcal O}
\newcommand{\LL}{\mathcal L}
\newcommand{\BB}{\mathcal B}
\newcommand{\EE}{\mathscr E}
\newcommand{\WW}{\mathscr W}
\newcommand{\MW}{\operatorname{MW}}
\newcommand{\NS}{\operatorname{NS}}
\newcommand{\Pic}{\operatorname{Pic}}
\newcommand{\rk}{\operatorname{rank}}
\newcommand{\Bl}{\operatorname{Bl}}

\newcommand{\ord}{\operatorname{ord}}
\newcommand{\Hom}{\operatorname{Hom}}

\newcommand{\dd}{\mathop{}\!\mathrm d}
\newcommand{\IV}{\textup{\textbf{IV}}}
\newcommand{\Ione}{\textup{\textbf{I}}_1}
\newcommand{\Aroot}{\mathrm A_2}
\newcommand{\Qquart}{\mathcal Q}
\newcommand{\Pp}[1]{P'_{#1}}

\allowdisplaybreaks[1]
\setlist[enumerate,1]{label=\textup{(\roman*)},leftmargin=*,itemsep=2pt,topsep=4pt}

\title[An elliptic Calabi--Yau threefold of rank twelve]{An elliptic Calabi--Yau threefold of Mordell--Weil rank twelve}
\author[D. Burek]{Dominik Burek}
\address{D. Burek\\Faculty of Mathematics and Computer Science\\Jagiellonian University, \L{}ojasiewicza 6\\30-348 Krak\'ow, Poland}
\email{dominik.burek@uj.edu.pl}
\date{}

\begin{document}
\begin{abstract}
We construct a smooth projective Calabi--Yau threefold with a flat elliptic
fibration of Mordell--Weil rank twelve and Hodge numbers
$(h^{1,1},h^{2,1})=(25,1)$.
The construction starts from the Kummer surfaces associated with the
products of a fixed elliptic curve and the level-four modular family.
The threefold is birational to an anticanonical quartic in
$\PP(\OO_{\PP^1}(3)^{\oplus2}\oplus\OO_{\PP^1}(4)\oplus\OO_{\PP^1})$.
We give twelve independent sections and construct a projective crepant
resolution of the associated Weierstrass model.
Varying the fixed elliptic curve gives a family of quartic resolutions
that is locally complete at the original member. The resulting fourfold
is birational over $\CC$ to a diagonal order-four quotient of a product
of four elliptic curves with $j$-invariant $1728$, and every smooth
projective model has Kodaira dimension zero. An algebraic correspondence
identifies the third cohomology of the original quartic resolution with
the tensor product of the first cohomology of the fixed elliptic curve
and the transcendental cohomology of the modular K3 surface. The Galois
representation on the latter factor is associated with a weight-three
CM newform.
\end{abstract}
\maketitle
\enlargethispage{-\baselineskip}

\section{Introduction}\label{sec:intro}

We construct an elliptic Calabi--Yau threefold of Mordell--Weil rank
twelve. For a parameter $z$, consider the elliptic curve
\begin{equation}\label{eq:Ez}
 E_z:\quad \eta^2=\xi\bigl(\xi+(z^2+1)^2\bigr)
                       \bigl(\xi+(z^2-1)^2\bigr).
\end{equation}
We keep $E_2$ fixed and let $E_v$ vary in \eqref{eq:Ez}.
For nonsingular $E_v$, the minimal resolutions of
$(E_2\times E_v)/\{\pm1\}$ form a family of Kummer K3 surfaces, where $-1$
acts on both factors.
The quadratic base change $r=t^2$ in their elliptic pencil, with parameter
$r$ as in Section~\ref{sec:origin}, gives the surfaces
$F^{(4)}_{E_2,E_v}$ of Kumar and Kuwata
\cite[Definition~2.3 and Section~5.4]{KK}. We use their quartic model and
explicit sections to construct an elliptic Calabi--Yau threefold.

Grassi and Weigand constructed smooth elliptic Calabi--Yau threefolds
of Mordell--Weil rank ten \cite{GW}. This is also the largest known rank
reported in \cite{GMPSW}. Theorem~6.1 of that paper gives an upper bound
of eighteen for a smooth elliptic Calabi--Yau threefold with singular
fibers and smooth rational base other than $\PP^2$. This bound applies
to the fibration constructed below, whose base is a blow-up of $\FF_3$.
To the best of our knowledge, the construction below gives the first
smooth projective elliptic Calabi--Yau threefold with Mordell--Weil rank
greater than ten.

To compactify the family of quartic models of these K3 surfaces, we use the twists
$\OO(3)^{\oplus2}\oplus\OO(4)\oplus\OO$ of the ambient projective bundle.
Write $[X_0:X_1:X_2:X_3]$ for its fiber coordinates.
The factors involving $X_0,X_1$ have constant coefficients because $E_2$ is
fixed; the coefficients $(v^2+1)^2$ and $(v^2-1)^2$ in the other two
linear factors have degree four in $v$.
These degrees determine the chosen twists, as explained in
Section~\ref{sec:origin}.
If $M$ denotes this projective bundle, $H=c_1(\OO_M(1))$, and $F$ is
the pullback of a point of $\PP^1$, the quartic has class
$4H+12F=-K_M$. Thus it is an anticanonical hypersurface in $M$.
Its six isolated singularities admit crepant weighted blow-ups. This gives
a smooth projective Calabi--Yau model $Y$.

To obtain a flat elliptic fibration, we use the birational Weierstrass
model. The elliptic parameter $t=X_1/X_3$ of the quartic pencil transforms as
a fiber coordinate on the Hirzebruch surface
\[
 \FF_3=\PP\bigl(\OO_{\PP^1}\oplus\OO_{\PP^1}(3)\bigr).
\]
The projection $\FF_3\to\PP^1$ is a $\PP^1$-bundle.
After blowing up six points of this base, the Weierstrass model has two
curves of transversal $\Aroot$ singularities and ninety-six ordinary double
points. The images of six explicit sections cover all the double points. These
images give global divisors whose blow-ups provide projective small
resolutions. Together with the resolution of the two singular curves,
this produces a flat elliptic fibration with a regular zero section.
We denote this flat model by $X$; it is birational to $Y$, which carries
the original non-flat fibration over $\FF_3$.

We give twelve independent sections, including the six used in the
resolution. The Picard number of the geometric generic K3 fiber gives the
matching upper bound for the rank. We compute the Picard number and
Euler characteristic of the threefold from the elliptic fibration, obtaining
$h^{1,1}=25$ and $h^{2,1}=1$.
In Section~\ref{sec:variation} we let the parameter of the fixed curve
vary. This gives a locally complete one-parameter deformation of the
quartic resolution at $u=2$. We identify its third cohomology with
$H^1(E_2)\otimes T(\mathcal S)$, where $\mathcal S$ is the modular K3
surface, and describe the resulting Galois representations over $\QQ(i)$.
The fourfold obtained by varying both elliptic curves has Kodaira
dimension zero; over $\CC$ it is birational to a diagonal order-four
quotient of a product of four elliptic curves with $j=1728$.

Let
\begin{equation}\label{eq:ABC}
 A(z)=z^8+14z^4+1,\qquad B(z)=z(z^4-1),\qquad
 C(z)=z^{12}-33z^8-33z^4+1.
\end{equation}
We shall repeatedly use the identity
\begin{equation}\label{eq:ABCidentity}
 A(z)^3-C(z)^2=108B(z)^4.
\end{equation}
Put
\begin{equation}\label{eq:constants}
 a=A(2)=481,\qquad b=B(2)=30,\qquad c=C(2)=-4879.
\end{equation}
Consider the elliptic curve over $\CC(v,t)$ given by
\begin{equation}\label{eq:mainW}
 \EE:\quad y^2=x^3-27aA(v)t^4x+
 2916\bigl(b^4t^{10}+B(v)^4t^2\bigr)+54cC(v)t^6.
\end{equation}
Let $S_-$ be the negative section of $\FF_3$, with $S_-^2=-3$, and let $S_+$ be a disjoint section of self-intersection $3$. We use $v$ for the parameter of the ruling and $t$ for the fiber coordinate, with $S_+=\{t=0\}$ and $S_-=\{t=\infty\}$. Define
\begin{equation}\label{eq:base}
 \Sigma=\{0,\infty,1,-1,i,-i\},\qquad
 \BB=\Bl_{\{(v,0):v\in\Sigma\}}\FF_3.
\end{equation}
\Needspace{18\baselineskip}
Our main result realizes \eqref{eq:mainW} as the generic fiber of a flat
elliptic Calabi--Yau threefold over $\BB$.
\begin{Thm}\label{thm:main}
There is a smooth projective threefold $X$ and a flat morphism
\[
 \pi:X\longrightarrow\BB
\]
with connected fibers and a regular section such that the following assertions hold.
\begin{enumerate}
\item $K_X\simeq\OO_X$ and $H^1(X,\OO_X)=H^2(X,\OO_X)=0$.
\item The generic fiber of $\pi$ is the elliptic curve \eqref{eq:mainW}. Its Mordell--Weil group over $\CC(v,t)$ is torsion-free of rank $12$.
\item The only reducible fibers over generic points of discriminant curves are two split fibers of type $\IV$.
\item The Hodge numbers and Euler characteristic are
\[
 h^{1,1}(X)=25,\qquad h^{2,1}(X)=1,\qquad e(X)=48.
\]
\end{enumerate}
\end{Thm}

Section~\ref{sec:origin} derives the quartic from the product of elliptic
curves and explains the choice of the ambient bundle.
Sections~\ref{sec:quartic} and~\ref{sec:resolution} describe the quartic
model and its crepant resolution $Y$. In Section~\ref{sec:MW} we identify
the elliptic pencil and compute its Mordell--Weil rank.
Sections~\ref{sec:W} and~\ref{sec:nodes} construct the flat model $X$, and
Section~\ref{sec:hodge} gives its Hodge numbers. The explicit sections,
height pairing and incidences with the ordinary double points are collected
in the appendices.

\subsection{Conventions}
All geometric statements are over $\CC$, unless a smaller field is explicitly specified. A Calabi--Yau threefold means a smooth projective threefold with trivial canonical bundle and $H^1(\OO)=H^2(\OO)=0$. We write $\PP(\mathcal V)$ for the bundle of one-dimensional subspaces of a vector bundle $\mathcal V$; thus $\OO_{\PP(\mathcal V)}(-1)$ is the tautological subbundle. We use additive notation for line bundles when writing divisor classes. The symbol $e$ denotes the topological Euler characteristic. Mordell--Weil groups are groups of rational sections, or equivalently groups of rational points on the generic elliptic fiber. The zero section is denoted by $O$.

\subsection*{Acknowledgements}
We thank S{\l}awomir Cynk and Matthias Sch\"utt for their careful reading
of the manuscript and valuable comments which helped to improve the
exposition. We also thank Matthias Sch\"utt for questions about the
variation of the fixed elliptic curve, the birational type of the
resulting fourfold, and the associated Galois representations.
The example arose from discussions with ChatGPT (OpenAI); ChatGPT and
Codex were also used to assist with symbolic computations and the
drafting and revision of the manuscript. The author is responsible for
the mathematical content.

\section{The Kumar--Kuwata quartic}\label{sec:origin}

\subsection{The elliptic curves and their invariants}
The family \eqref{eq:Ez} is the model in
\cite[Section~5.1]{KK}. The full two-torsion is visible in the three roots of the cubic. Moreover,
all pairwise differences of these roots are squares in $\CC(z)$:
they are, up to sign, $(z^2+1)^2$, $(z^2-1)^2$, and $4z^2$.
The usual halving formulas for a cubic with three rational roots therefore
give its full four-torsion over $\CC(z)$. For instance, the points above
$\xi=(z^2+1)(z^2-1)$ have
$\eta=\pm2z^2(z^4-1)$ and double to $(0,0)$.
The sections used below have coordinates in $\QQ(i)(v,t)$; their
formulas are given in Appendix~\ref{app:sections}.

Write the cubic as
$\xi^3+a_2\xi^2+a_4\xi$, where
\[
 a_2=2(z^4+1),\qquad a_4=(z^4-1)^2.
\]
The formulas $c_4=16(a_2^2-3a_4)$ and
$c_6=-64a_2^3+288a_2a_4$ give
\begin{equation}\label{eq:modularinvariants}
 c_4(E_z)=16A(z),\qquad c_6(E_z)=64C(z),\qquad
 \Delta(E_z)=256B(z)^4.
\end{equation}
Here the last equality also follows from
$\Delta=16a_4^2(a_2^2-4a_4)$.
Consequently $c_4^3-c_6^2=1728\Delta$ is precisely
\eqref{eq:ABCidentity}, and
\begin{equation}\label{eq:jEz}
 j(E_z)=\frac{16A(z)^3}{B(z)^4}.
\end{equation}

The homogeneous forms corresponding to $A,B,C$ have degrees $8,6,12$.
For $B$ the homogenization is $UV(U^4-V^4)$, so its zeros are precisely
$\Sigma=\{0,\infty,1,-1,i,-i\}$, as in \eqref{eq:base}.
The invariant $c_4$ is nonzero there and the minimal discriminant has
order four at each point. Thus the minimal elliptic surface
$\varphi:\mathcal S\to\PP^1_z$ has six fibers of type $\mathrm I_4$.
Its \emph{fundamental line bundle} is the line bundle on $\PP^1_z$
\[
 \mathcal L=(R^1\varphi_*\OO_{\mathcal S})^\vee
       \simeq\varphi_*\omega_{\mathcal S/\PP^1}.
\]
The isomorphism is relative duality: over a smooth fiber $E_z$, its
fiber is $H^0(E_z,\omega_{E_z})$. This bundle also determines how the
local Weierstrass equations glue: $a_j$ is a section of $\mathcal L^j$.
Here the coordinates at infinity are $\xi_\infty=w^4\xi$ and
$\eta_\infty=w^6\eta$, where $w=1/z$, so
$\mathcal L\simeq\OO_{\PP^1}(2)$. Accordingly, $a_2$ and $a_4$
homogenize to degrees four and eight.

The canonical bundle formula \cite[Theorem~6.8]{SS} gives
\[
 \omega_{\mathcal S}\simeq
 \varphi^*(\omega_{\PP^1}\otimes\mathcal L)\simeq\OO_{\mathcal S}.
\]
Moreover, $R^1\varphi_*\OO_{\mathcal S}\simeq\OO_{\PP^1}(-2)$,
so the Leray sequence gives $H^1(\mathcal S,\OO_{\mathcal S})=0$.
Thus $\mathcal S$ is a K3 surface.

In fact its full Mordell--Weil group is
\[
 E_z(\CC(z))\simeq(\ZZ/4\ZZ)^2.
\]
Indeed, the six fibers contribute $6\cdot3$ to the Shioda--Tate formula,
so $\rho(\mathcal S)=20$ and the rank is zero. The component groups have
exponent four. Hence the multiple $4P$ of a torsion section meets the
identity component of every fiber; the height formula then forces
$4P=O$. The full four-torsion exhibited above is therefore the whole
group.

\subsection{From a product of curves to the quartic}
Let $u,v$ be independent parameters and put
$h_z(T)=T(T+(z^2+1)^2)(T+(z^2-1)^2)$.
On $E_u\times E_v$, simultaneous negation of the two elliptic coordinates
fixes $\xi,\zeta$ and the ratio $r=\eta_v/\eta_u$. The invariant function
field has equation
\begin{equation}\label{eq:originKummer}
 h_v(\zeta)=r^2h_u(\xi).
\end{equation}
Its minimal smooth model is the Kummer surface of $E_u\times E_v$.
Replace $r$ by $t^2$. With
$\xi=X_0/X_1$, $\zeta=X_2/X_3$, and $t=X_1/X_3$, the resulting equation is
\begin{equation}\label{eq:originquartic}
 X_2X_3\bigl(X_2+(v^2+1)^2X_3\bigr)\bigl(X_2+(v^2-1)^2X_3\bigr)
 =X_0X_1\bigl(X_0+(u^2+1)^2X_1\bigr)\bigl(X_0+(u^2-1)^2X_1\bigr).
\end{equation}
Indeed the two cubic denominators contribute $X_3^3$ and $X_1^3$, and
$t^4=(X_1/X_3)^4$ supplies the remaining factor $X_1/X_3$.
This is the quartic model of
\cite[Section~5.4, equation~(5.2)]{KK}; its pencil is $t=X_1/X_3$.
Generically it contains exactly sixteen lines, as described there.
They are the intersections of one of the four planes
\[
 X_0=0,\quad X_1=0,\quad X_0+(u^2+1)^2X_1=0,\quad
 X_0+(u^2-1)^2X_1=0
\]
with one of the four planes
\[
 X_2=0,\quad X_3=0,\quad X_2+(v^2+1)^2X_3=0,\quad
 X_2+(v^2-1)^2X_3=0.
\]
The planes $X_1=tX_3$ contain the line $X_1=X_3=0$.
For $u=2$, removing this fixed line from their intersections with the
quartic leaves the plane cubic written in \eqref{eq:cubic}.
For $u=2$ the factors $(u^2+1)^2$ and $(u^2-1)^2$ are $25$ and $9$,
as in
\eqref{eq:quarticcoeff}.

Following \cite[Definitions~2.1 and~2.3]{KK}, we write $F^{(n)}$ for the
elliptic surface obtained from the Inose surface $F^{(1)}$ by the cyclic
base change $t_1=t_n^n$. In this notation \eqref{eq:originKummer} is the $F^{(2)}$
pencil and \eqref{eq:originquartic} is $F^{(4)}$.
For $u=2$, Proposition~\ref{prop:genericidentification} identifies
the generic member of this residual cubic pencil with \eqref{eq:mainW}
over $\CC(v,t)$, taking $X_0=X_2=0$ as the zero section.

\subsection{The choice of the projective bundle}
Write $p:M\to\PP^1$ for the ambient projective bundle to be chosen.
By base degree $d$ for a fiber coordinate we mean the transition rule
$X_j'=w^dX_j$ on the overlap $w=1/v$. Equivalently, on the projective
bundle it is a section of $\OO_M(1)\otimes p^*\OO_{\PP^1}(d)$.
The coefficients $(v^2\pm1)^2$ have degree four. Give $X_3$ base degree
zero and $X_2$ base degree four, so that
$X_2+(v^2\pm1)^2X_3$ is homogeneous of base degree four.
The left side of \eqref{eq:originquartic} then has base degree twelve.
Once $u$ is fixed, the four factors on the right have constant
coefficients. Giving $X_0$ and $X_1$ base degree three therefore
balances the equation. This gives the ambient bundle
\[
 M=\PP\bigl(\OO_{\PP^1}(3)^{\oplus2}\oplus
             \OO_{\PP^1}(4)\oplus\OO_{\PP^1}\bigr),
 \qquad p:M\to\PP^1.
\]
Let $H=c_1(\OO_M(1))$ and let $F$ be the pullback of a point of $\PP^1$.
The equation has fiber degree four and base degree twelve, so its class
is $4H+12F$. The canonical bundle formula gives
\[
 -K_M=4H+(3+3+4+0+2)F=4H+12F.
\]
Thus the quartic is anticanonical in $M$.
Proposition~\ref{prop:anticanonical} checks the equation at infinity
and applies adjunction to its dualizing sheaf.

The difference between the base degrees of $X_1$ and $X_3$ also explains
the ruled base: $t=X_1/X_3$ transforms by $t'=w^3t$, so it is a fiber
coordinate on $\FF_3$.

\section{An anticanonical quartic model}\label{sec:quartic}

Set
\begin{equation}\label{eq:quarticcoeff}
 \alpha(v)=(v^2+1)^2,\qquad \beta(v)=(v^2-1)^2,
 \qquad q(X_0,X_1)=X_0X_1(X_0+25X_1)(X_0+9X_1).
\end{equation}
Let
\[
 M=\PP\bigl(\OO_{\PP^1}(3)\oplus\OO_{\PP^1}(3)\oplus\OO_{\PP^1}(4)\oplus\OO_{\PP^1}\bigr)
\]
with projection $p:M\to\PP^1$ and fiber coordinates $[X_0:X_1:X_2:X_3]$. On the affine chart with base parameter $v$, define $\Qquart\subset M$ by
\begin{equation}\label{eq:quartic}
 X_2X_3\bigl(X_2+\alpha(v)X_3\bigr)\bigl(X_2+\beta(v)X_3\bigr)=q(X_0,X_1).
\end{equation}

\begin{Pro}\label{prop:anticanonical}
The hypersurface $\Qquart\subset M$ defined by \eqref{eq:quartic}
is integral and anticanonical. Its dualizing sheaf is trivial.
\end{Pro}
\begin{proof}
On the chart $w=1/v$, use the transition functions
\begin{equation}\label{eq:bundletransition}
 X_0'=w^3X_0,\qquad X_1'=w^3X_1,\qquad X_2'=w^4X_2,\qquad X_3'=X_3.
\end{equation}
Since $w^4\alpha(1/w)=\alpha(w)$ and $w^4\beta(1/w)=\beta(w)$, both sides of \eqref{eq:quartic} are multiplied by $w^{12}$. Thus the equation is global, including the fiber over infinity.

Let $H=c_1(\OO_M(1))$ be the relative hyperplane class and let $F$ be the pullback of a point of $\PP^1$. The coordinate divisors have classes $H+3F,H+3F,H+4F,H$. Consequently
\[
 [\Qquart]=4H+12F.
\]
With our convention for projectivization,
\[
 K_M=-4H+p^*\bigl(K_{\PP^1}-c_1(\OO(3)^{\oplus2}\oplus\OO(4)\oplus\OO)\bigr)
 =-4H-12F.
\]
The hypersurface $\Qquart$ is an effective Cartier divisor in the
smooth fourfold $M$, and hence is Gorenstein. Adjunction for its
dualizing sheaf gives
\[
 \omega_{\Qquart}\simeq
 \bigl(\omega_M\otimes\OO_M(\Qquart)\bigr)|_{\Qquart}
 \simeq\OO_{\Qquart}.
\]
This does not require $\Qquart$ to be smooth. The general fiber is a smooth quartic surface, by the calculation below, hence is integral. No entire fiber of $M\to\PP^1$ is contained in \eqref{eq:quartic}, since $q$ is not the zero polynomial. It follows that $\Qquart$ is integral.
\end{proof}

\begin{Lem}\label{lem:sixsing}
The singular locus of $\Qquart$ consists of six points, one over each element of $\Sigma$. Outside these six fibers, the map $\Qquart\to\PP^1_v$ is smooth with quartic K3 fibers.
\end{Lem}
\begin{proof}
The binary quartic $q$ has four distinct projective zeros. Its two partial derivatives can therefore vanish simultaneously only at $X_0=X_1=0$. The other binary quartic in \eqref{eq:quartic} has roots
\[
 X_3=0,\qquad X_2=0,\qquad X_2=-\alpha(v)X_3,\qquad X_2=-\beta(v)X_3.
\]
A repeated root occurs precisely when $\alpha=0$, $\beta=0$, or $\alpha=\beta$. On the finite chart this gives $v=\pm i$, $v=\pm1$, and $v=0$, respectively. The chart \eqref{eq:bundletransition} gives the additional value $v=\infty$.

At each of these values precisely two roots coincide. Their difference vanishes to order two in the base parameter. The base derivative of the defining equation also vanishes at the collision point, so it is a singular point of the threefold. There is no collision at $X_3=0$. The same partial derivative argument excludes all other singular points, both on the total space and on its fibers.
\end{proof}

Since $\Qquart$ is a hypersurface with isolated singularities, it is normal. We shall use the notation $\Qquart_{v_0}$ for its quartic fiber over $v_0$.

\section{A crepant resolution of the quartic model}\label{sec:resolution}

Each singularity in Lemma~\ref{lem:sixsing} is resolved by a weighted
blow-up with weights $(1,1,1,2)$. We begin with the local calculation.

\begin{Lem}\label{lem:weightedblowup}
Let a hypersurface germ in a smooth fourfold have local coordinates $(x,y,s,z)$ and an equation whose lowest weighted part, for weights $(1,1,1,2)$, is
\begin{equation}\label{eq:weightedinitial}
 \lambda z^2+\mu s^4-q(x,y),\qquad \lambda\mu\ne0,
\end{equation}
where $q$ is a binary quartic with four distinct projective zeros. The weighted blow-up of weights $(1,1,1,2)$ has smooth strict transform and is crepant. Its exceptional divisor is a smooth del Pezzo surface of degree two. This blow-up is projective and can be realized algebraically.
\end{Lem}
\begin{proof}
The weighted order of the equation is four. The discrepancy of the strict transform is
\begin{equation}\label{eq:weighteddiscrepancy}
 (1+1+1+2)-1-4=0.
\end{equation}
The exceptional divisor is
\[
 D=\{\lambda z^2+\mu s^4-q(x,y)=0\}\subset\PP(1,1,1,2).
\]
Projection to $[x:y:s]$ identifies $D$ with a double cover of $\PP^2$ branched along $q(x,y)-\mu s^4=0$. A singular point of this plane quartic would have $s=0$ and $q_x=q_y=0$, which is impossible projectively. Thus the branch curve is smooth, and $D$ is a smooth del Pezzo surface with $K_D^2=2$.

We check smoothness in the weighted blow-up charts. In each weight-one chart, divide the equation by the fourth power of the exceptional parameter. On the exceptional divisor the resulting equation is a dehomogenization of \eqref{eq:weightedinitial}, and its gradient is nonzero. The strict transform is therefore smooth there. In the weight-two chart, the index cover has coordinates $(r,x_1,y_1,s_1)$ with
\[
 z=r^2,\qquad x=rx_1,\qquad y=ry_1,\qquad s=rs_1,
\]
and the residual $\mu_2$ acts by changing all four signs. On $r=0$ the equation is $\lambda+\mu s_1^4-q(x_1,y_1)=0$, whose gradient is nonzero by the preceding argument. The only fixed point of $\mu_2$ is the origin, where the divided equation takes the value $\lambda\ne0$. Thus the strict transform and its quotient are smooth on this chart as well. Away from the exceptional divisor the morphism is an isomorphism.

In a smooth algebraic neighborhood, the weighted blow-up is the normalization of the blow-up of
\[
 (z)+(x,y,s)^2.
\]
Indeed, this monomial ideal has the unique compact Newton facet $m_x+m_y+m_s+2m_z=2$. The normalized blow-up has the star subdivision determined by the primitive vector $(1,1,1,2)$. This also proves projectivity. The strict transform of the hypersurface gives the claimed resolution.
\end{proof}

\begin{Pro}\label{prop:Qresolution}
The six singularities of $\Qquart$ admit disjoint projective crepant resolutions as in Lemma~\ref{lem:weightedblowup}. In particular, there is a smooth projective threefold
\[
 \rho:Y\longrightarrow\Qquart
\]
with exceptional divisors $D_1,\ldots,D_6$, each a degree-two del Pezzo surface, and $K_Y\simeq\OO_Y$.
\end{Pro}
\begin{proof}
At $v=0$, work on $X_3=1$, put $x=X_0$, $y=X_1$, and set $s=v$ and $z=X_2+1+s^4$. The exact local equation is
\begin{equation}\label{eq:localzero}
 (z-1-s^4)(z^2-4s^4)-q(x,y)=0.
\end{equation}
Its weighted initial form is $-z^2+4s^4-q(x,y)$.

At $v_0=\pm1$, put $s=v-v_0$ and $z=X_2+\beta(v)/2$. The equation is
\[
 \bigl(z+\alpha(v)-\beta(v)/2\bigr)\bigl(z^2-\beta(v)^2/4\bigr)-q(x,y)=0.
\]
Here $\alpha(v_0)=4$ and $\beta(v)=4s^2+O(s^3)$, so the initial form is
\begin{equation}\label{eq:localone}
 4z^2-16s^4-q(x,y).
\end{equation}
At $v_0=\pm i$, interchange $\alpha$ and $\beta$. Since the square of the leading coefficient of $\alpha(v)$ is again $16$, the initial form is also \eqref{eq:localone}. The point over infinity is identical to the point over zero in the chart \eqref{eq:bundletransition}.

Lemma~\ref{lem:weightedblowup} applies to all six germs. Their coordinate ideals have support at distinct points and extend by the unit ideal to coherent ideals on the smooth ambient fourfold. The normalized ambient blow-ups, followed by strict transform, are therefore global projective morphisms. The centers are disjoint, and no additional singularity remains. Crepantness and Proposition~\ref{prop:anticanonical} give $K_Y\simeq\OO_Y$.
\end{proof}

\Needspace{6\baselineskip}
\begin{Pro}\label{prop:QCY}
The threefold $Y$ is Calabi--Yau.
\end{Pro}
\begin{proof}
The projective bundle $M$ satisfies $H^j(M,\OO_M)=0$ for every $j>0$. Since $\dim M=4$, Serre duality gives $H^j(M,K_M)=0$ for $j<4$. The exact sequence
\[
 0\longrightarrow\OO_M(K_M)\longrightarrow\OO_M
 \longrightarrow\OO_{\Qquart}\longrightarrow0
\]
therefore yields $H^1(\Qquart,\OO_{\Qquart})=H^2(\Qquart,\OO_{\Qquart})=0$.

The normal Gorenstein variety $\Qquart$ has a smooth crepant resolution, so its singularities are canonical. In characteristic zero they are rational; see \cite[5.22]{KM}. Thus $R^j\rho_*\OO_Y=0$ for $j>0$, and the two vanishings hold on $Y$. The canonical bundle is trivial by Proposition~\ref{prop:Qresolution}.
\end{proof}

\section{The elliptic pencil and its Mordell--Weil group}\label{sec:MW}

\subsection{The elliptic pencil}
The planes through the line $X_1=X_3=0$ cut out that fixed line and a residual
plane cubic. The quotient $X_1/X_3$ labels the planes. We first extend this
quotient across the fixed line, and then compute a pointed Weierstrass
equation for the residual cubic.
On \eqref{eq:quartic}, put $t=X_1/X_3$. Write
\[
 G=X_2(X_2+\alpha X_3)(X_2+\beta X_3),\qquad
 H_0=X_0(X_0+25X_1)(X_0+9X_1).
\]
The equation is $X_3G=X_1H_0$. Along $X_1=X_3=0$, the pair $[G:H_0]$ equals $[X_2^3:X_0^3]$ and is defined. Thus $[X_1:X_3]$ and $[G:H_0]$ glue to a morphism. Under \eqref{eq:bundletransition}, one has $t'=w^3t$, so the target is $\FF_3$.

After removing the fixed line from the plane section $X_1=tX_3$, the residual cubic is
\begin{equation}\label{eq:cubic}
 X_2(X_2+\alpha X_3)(X_2+\beta X_3)=tX_0(X_0+25tX_3)(X_0+9tX_3).
\end{equation}
It contains the point $X_0=X_2=0$. The following calculation identifies its generic fiber over $\CC(v,t)$.

The next lemma is the standard quartic-to-Weierstrass transformation;
see \cite{ConnellAddendum}.
We give the formulas to specify
the normalization and the image of the chosen zero.

\begin{Lem}\label{lem:quarticjacobian}
Let a genus-one double cover with square-free quartic equation over a characteristic-zero field be given by
\[
 \eta^2=a_4s^4+a_3s^3+a_2s^2+a_1s+a_0.
\]
Suppose it has a rational point with finite $s$-coordinate and nonzero $\eta$-coordinate. Set
\[
 I=12a_4a_0-3a_3a_1+a_2^2,\qquad
 J=72a_4a_2a_0+9a_3a_2a_1-27a_4a_1^2-27a_3^2a_0-2a_2^3.
\]
With the indicated point as zero, its elliptic curve is isomorphic over the same field to
\begin{equation}\label{eq:binaryjacobian}
 V^2=U^3-27IU-27J.
\end{equation}
\end{Lem}
\begin{proof}
Translate $s$ so that the point is $(0,q)$, with $q\ne0$ and $a_0=q^2$; the displayed invariants are unchanged by this translation. Put
\[
 u=\frac{2q(\eta+q)+a_1s}{s^2},\qquad
 z=\frac{(u^2-4q^2a_4)s-(a_1u+2q^2a_3)}{2q}.
\]
Substitution gives
\[
 z^2=u^3+a_2u^2+(a_3a_1-4a_4q^2)u+
       a_4a_1^2+q^2a_3^2-4a_4a_2q^2.
\]
The formula defining $z$ gives $s$ as a rational function of $u,z$; substitution in the definition of $u$ then recovers $\eta$. Hence the map is birational. The change
$U=9u+3a_2$, $V=27z$ gives \eqref{eq:binaryjacobian}. The chosen point maps to infinity. Birational maps between smooth projective curves extend to isomorphisms, proving the assertion.
\end{proof}

\begin{Pro}\label{prop:genericidentification}
With $X_0=X_2=0$ as zero, the generic cubic \eqref{eq:cubic} is isomorphic over $\CC(v,t)$ to \eqref{eq:mainW}.
\end{Pro}
\begin{proof}
Put $\xi=X_0/(tX_3)$ and $\zeta=X_2/X_3$. Equation~\eqref{eq:cubic} becomes
\[
 \zeta(\zeta+\alpha)(\zeta+\beta)=t^4\xi(\xi+25)(\xi+9).
\]
Projection from $(0,0)$, with $\zeta=k\xi$, gives the double-cover equation
\begin{equation}\label{eq:projectionquartic}
 \eta^2=\bigl((\alpha+\beta)k^2-34t^4\bigr)^2
       -4(k^3-t^4)(\alpha\beta k-225t^4).
\end{equation}
The original point maps to
\[
 k_0=\frac{225t^4}{\alpha\beta},\qquad
 \eta_0=(\alpha+\beta)k_0^2-34t^4\ne0.
\]
The coefficients of the binary quartic in \eqref{eq:projectionquartic}, in decreasing order, are
\[
 16v^4,\quad900t^4,\quad-136(v^4+1)t^4,\quad
 4(v^4-1)^2t^4,\quad256t^8.
\]
Their invariants are
\begin{equation}\label{eq:projectionIJ}
 I=16aA(v)t^8,\qquad
 J=-128t^{12}\left(54b^4t^4+cC(v)+54B(v)^4/t^4\right).
\end{equation}
These identities follow by substituting the five coefficients into the formulas in Lemma~\ref{lem:quarticjacobian}. The quartic has nonzero discriminant over $\CC(v,t)$: $4I^3-J^2$ has a nonzero leading term in $t$. Thus its smooth projective model has genus one, and so does the birational plane cubic, which is consequently smooth. Apply Lemma~\ref{lem:quarticjacobian} and set $U=4t^2x$, $V=8t^3y$. Equation~\eqref{eq:binaryjacobian} becomes precisely \eqref{eq:mainW}. Every transformation is over the original function field and preserves the chosen zero.
\end{proof}

\begin{Rem}\label{rem:nonflat}
Write $\Sigma=\{v_1,\ldots,v_6\}$ for the six parameters at which
the quartic K3 fibration has singular fibers. The morphism $Y\to\FF_3$
contracts the exceptional del Pezzo surface over $v_i$ to $(v_i,0)$.
It therefore has two-dimensional fibers over these six points and is
not flat. In Section~\ref{sec:W} we construct
a birational Weierstrass model over $\BB$; its resolution will carry a flat
elliptic fibration with a regular zero section.
\end{Rem}

\subsection{The Picard number of the generic quartic}
We use Inose's comparison of the Picard numbers of $F^{(4)}$ and its
associated Kummer surface; see \cite[Corollary~1.2]{Inose} and
\cite[Remark~2.4]{KK}.

\begin{Lem}\label{lem:genericPicard}
The geometric generic quartic fiber $S$ of $\Qquart\to\PP^1_v$ has Picard number $18$.
\end{Lem}
\begin{proof}
The elliptic curves $E_2$ and $E_v$ are not isogenous over
$\overline{\CC(v)}$. Otherwise an isogeny of minimal degree would have
cyclic kernel and would imply
\[
 \Phi_N\bigl(j(E_v),j(E_2)\bigr)=0
\]
for a classical modular polynomial $\Phi_N$. Since
$\Phi_N(T,j(E_2))$ is monic in $T$, this contradicts the nonconstancy
of $j(E_v)$ in \eqref{eq:jEz}.

By Section~\ref{sec:origin}, $S$ is the surface $F^{(4)}_{E_2,E_v}$.
Inose's theorem gives
\[
 \rho(S)=\rho\bigl(\operatorname{Km}(E_2\times E_v)\bigr).
\]
The product has N\'eron--Severi rank
$2+\rk\Hom(E_2,E_v)=2$. Simultaneous negation acts trivially on its
second cohomology, and the Kummer resolution adds sixteen independent
exceptional classes. Hence $\rho(S)=2+16=18$.
\end{proof}

For $t\ne0$, the substitution $x=t^2x_0$, $y=t^3y_0$ in \eqref{eq:mainW} gives
\begin{equation}\label{eq:nonminimal}
 y_0^2=x_0^3-27aA(v)x_0+
 54\left(54b^4t^4+cC(v)+\frac{54B(v)^4}{t^4}\right).
\end{equation}

\Needspace{6\baselineskip}
\begin{Lem}\label{lem:K3fibers}
Over $\overline{\CC(v)}$, the elliptic fibration on $S$ has two split fibers of type $\IV$, at $t=0,\infty$, and sixteen fibers of type $\Ione$.
\end{Lem}
\begin{proof}
Write \eqref{eq:mainW} as $y^2=x^3+fx+g$, with discriminant $\Delta=-16(4f^3+27g^2)$. At $t=0$ the orders of $(f,g,\Delta)$ are $(4,2,4)$, and the leading coefficient of $g/t^2$ is $(54B(v)^2)^2$. The same calculation at infinity gives leading coefficient $(54b^2)^2$. Resolving the transversal $\Aroot$ point gives a split $\IV$ fiber in each case.

Put $q=t^4$, $k=aA(v)$, $d=54b^4$, $h=54B(v)^4$, and $p=cC(v)$. The remaining discriminant is a nonzero constant times
\[
 R(t)=(dt^8+pt^4+h)^2-k^3t^8.
\]
It has degree sixteen and nonzero constant term. A repeated nonzero root would satisfy
$(p+dq+h/q)^2=k^3$ and $d-h/q^2=0$. Hence $q=\varepsilon B(v)^2/b^2$ for $\varepsilon=\pm1$. But \eqref{eq:ABCidentity} gives
\begin{equation}\label{eq:nodekeyidentity}
 \bigl(cC(v)+108\varepsilon b^2B(v)^2\bigr)^2-(aA(v))^3
 =-108\bigl(b^2C(v)-\varepsilon cB(v)^2\bigr)^2,
\end{equation}
which is nonzero in $\CC(v)$. Thus the sixteen roots are simple. At each of them $f\ne0$, so the fiber is nodal of type $\Ione$.
\end{proof}

\begin{Pro}\label{prop:rank12}
The group $\EE(\CC(v,t))$ is torsion-free of rank twelve. Twelve independent points are defined over $\QQ(i)(v,t)$.
\end{Pro}
\begin{proof}
The twelve identities in Appendix~\ref{app:sections} define points over $\QQ(i)(v,t)$. On the restriction $v=5$, their height matrix is positive definite with determinant $4096/9$. A relation among the original points specializes to the same relation over $\CC(t)$: the elliptic curve has good reduction at the valuation $v-5$ of $\CC(t)(v)$, since its specialized discriminant is not identically zero. Thus the original twelve points are independent.

For the upper bound, restrict divisors on the geometric generic K3 to its elliptic generic fiber. The degree part is generated by the zero section, and the kernel of restriction is generated by vertical divisors. This gives the Shioda--Tate formula, and Lemmas~\ref{lem:genericPicard} and \ref{lem:K3fibers} give
\[
 \rk\EE(\overline{\CC(v)}(t))=18-2-2-2=12.
\]
The subgroup over $\CC(v,t)$ therefore has rank exactly twelve.

To exclude torsion, apply the height pairing. On a K3 elliptic surface with two $\IV$ fibers, Shioda's height formula \cite[Theorem~8.6]{Shioda} gives, for every nonzero section $P$,
\[
 \langle P,P\rangle=4+2P\cdot O-
   \operatorname{contr}_0(P)-\operatorname{contr}_\infty(P)
 \ge 4-\frac23-\frac23=\frac83>0.
\]
Here $\operatorname{contr}_0(P)$ and $\operatorname{contr}_\infty(P)$ are the local height corrections at the two fibers of type $\IV$. Each equals $0$ if $P$ meets the component met by $O$, and $2/3$ otherwise. Since $P\ne O$, one has $P\cdot O\ge0$. A torsion section has height zero, so there is no nonzero geometric torsion and hence no torsion over $\CC(v,t)$.
\end{proof}

\section{A Weierstrass model over a rational surface}\label{sec:W}

We blow up the six points of $\FF_3$ identified in Remark~\ref{rem:nonflat} and construct
a Weierstrass model over the resulting surface $\BB$.

Write \eqref{eq:mainW} as $y^2=x^3+fx+g$. The coefficients are
\begin{equation}\label{eq:fg}
 \begin{split}
 f&=-27aA(v)t^4,\\
 g&=2916\bigl(b^4t^{10}+B(v)^4t^2\bigr)+54cC(v)t^6.
 \end{split}
\end{equation}
Let $f_0$ denote a ruling fiber on $\FF_3$. Then
\[
 -K_{\FF_3}=2S_-+5f_0.
\]
The transition in the $v$ direction is
\begin{equation}\label{eq:Wtransition}
 w=v^{-1},\qquad t'=w^3t,\qquad x'=w^{10}x,\qquad y'=w^{15}y.
\end{equation}
It preserves the form of \eqref{eq:fg}, by \eqref{eq:ABC}. At $t=\infty$ put $r=t^{-1}$, $x_\infty=r^4x$, $y_\infty=r^6y$. Thus
\begin{equation}\label{eq:Winfty}
 f_\infty=-27aA(v)r^4,\qquad
 g_\infty=2916b^4r^2+54cC(v)r^6+2916B(v)^4r^{10}.
\end{equation}
Equations \eqref{eq:Wtransition} and \eqref{eq:Winfty} show that $f$ and $g$ are sections of $\LL^4$ and $\LL^6$, where $\LL=\OO_{\FF_3}(-K_{\FF_3})$.
This is the fundamental line bundle of the Weierstrass model: the local
coordinates $x,y$ transform with $\LL^2,\LL^3$, respectively.

The projective Weierstrass cubic $\WW_0\to\FF_3$ lies in
$\PP(\OO\oplus\LL^2\oplus\LL^3)$, with equation
\[
 y^2z=x^3+fxz^2+gz^3.
\]
Denote the ambient projection to $\FF_3$, and its restriction to $\WW_0$, by $\varpi$. If $H$ is the tautological divisor class, this cubic has class $3H+6\LL$, whereas the ambient canonical class is
$-3H+\varpi^*(K_{\FF_3}-5\LL)$. Adjunction therefore gives
\[
 \omega_{\WW_0}=\varpi^*(\omega_{\FF_3}\otimes\LL)\simeq\OO_{\WW_0}.
\]
The zero section is $[x:y:z]=[0:1:0]$. On $y=1$ the derivative of the equation in $z$ is one along it, so the total space is smooth there. Every fiber is a nonzero plane cubic, and the generic fiber is smooth and integral. Thus the total space is integral; its normality follows also from the codimension of its singular locus determined below.

\subsection{The six base blow-ups}
The homogeneous form associated with $B(v)$ has degree six and is
\[
 UV(U^4-V^4).
\]
Its six simple zeros are exactly $\Sigma$. At the points $(v,0)$ with $v\in\Sigma$, the orders of $(f,g)$ are $(4,6)$. Blow up these points and divide the pulled-back coefficients by the fourth and sixth powers of the exceptional equations. This yields a Weierstrass model
\[
 \WW\longrightarrow\BB.
\]
Let $\sigma:\BB\to\FF_3$ be the blow-up morphism and $E_1,\ldots,E_6$ its exceptional curves. The fundamental line bundle of $\WW$ is
\begin{equation}\label{eq:Lbaseblow}
 \LL_{\BB}=\sigma^*\LL-\sum_{i=1}^6E_i=-K_{\BB}.
\end{equation}
In particular, $\omega_{\WW}\simeq\OO_{\WW}$.

\begin{Lem}\label{lem:boundary}
Over the union of the strict transforms of $S_+,S_-$ and the exceptional curves $E_1,\ldots,E_6$, the only singularities of $\WW$ are two disjoint smooth curves of transversal split $\Aroot$ singularities. Their images are the strict transforms of $S_+$ and $S_-$, both of self-intersection $-3$. No other singularity occurs over an exceptional curve of $\BB\to\FF_3$.
\end{Lem}
\begin{proof}
At a finite zero $v_0$ of $B$, put $s=v-v_0$ and $B(v_0+s)=s\delta(s)$, with $\delta(0)\ne0$. On the base blow-up chart $t=e$, $s=ez$, the coefficients after minimalization restrict at $e=0$ to
\begin{equation}\label{eq:firstEchart}
 f_E=-27aA(v_0),\qquad
 g_E=2916\delta(0)^4z^4+54cC(v_0).
\end{equation}
If $z\ne0$, the derivative $\partial g_E/\partial z$ is nonzero, so the total space is smooth. At $z=0$, the discriminant is nonzero: $C(v_0)^2=A(v_0)^3$, while $a^3-c^2=108b^4\ne0$.

In the other chart $s=e$, $t=ez$, the coefficients are
\begin{equation}\label{eq:secondEchart}
 \begin{split}
 f'&=-27aA(v_0+e)z^4,\\
 g'&=2916\bigl(b^4e^4z^{10}+\delta(e)^4z^2\bigr)
       +54cC(v_0+e)z^6.
 \end{split}
\end{equation}
At $z=0$, the coefficient of $z^2$ in $g'$ is a unit and is a square. The local form is therefore
\begin{equation}\label{eq:A2local}
 y^2=x^3+r^4a_1(r,s)x+r^2b_1(r,s),\qquad b_1(0,s)\ne0.
\end{equation}
This is a transversal $\Aroot$ singularity. The chart over $v=\infty$ is identical by \eqref{eq:Wtransition}. Formula~\eqref{eq:Winfty} gives the same form along $S_-$, with leading coefficient $2916b^4$.

The two curves are disjoint. The self-intersection of $S_+$ drops from $3$ to $3-6=-3$, and $S_-$ is unchanged. The leading coefficients $2916\delta^4$ and $2916b^4$ are squares, so the two exceptional components of each transversal $\Aroot$ are individually defined. Finally, after removing the factor $r^4$ from the discriminant in \eqref{eq:A2local}, its restriction at $r=0$ is a nonzero multiple of $b_1(0,s)^2$. Thus no residual discriminant branch meets either curve, and the charts above account for the whole boundary.
\end{proof}

\begin{Lem}\label{lem:A2resolve}
Blowing up the two singular curves of Lemma~\ref{lem:boundary} resolves them projectively and crepantly. The zero section is unchanged.
\end{Lem}
\begin{proof}
In \eqref{eq:A2local}, blow up $(x,y,r)$ in the smooth four-dimensional ambient chart. The center has codimension three and the hypersurface multiplicity is two. For the ambient blow-up $b:\widetilde M\to M$ with exceptional divisor $E$, one has
$K_{\widetilde M}=b^*K_M+2E$ and $\widetilde W=b^*W-2E$. Adjunction therefore gives $K_{\widetilde W}=b^*K_W$; the discrepancy is $3-1-2=0$.

In the $r$-chart, with $x=rx_1$, $y=ry_1$, the equation becomes
\[
 y_1^2=rx_1^3+r^3a_1(r,s)x_1+b_1(r,s).
\]
On $r=0$, the derivative in $y_1$ is nonzero. In the $x$-chart, $y=xy_1$, $r=xr_1$, it becomes
\[
 y_1^2=x+x^3r_1^4a_1(xr_1,s)+r_1^2b_1(xr_1,s).
\]
At the only points not already covered by the $r$-chart, namely $r_1=0$ on the exceptional divisor, one has $y_1=0$ and the derivative in $x$ is nonzero. The remaining chart gives no further point. This proves smoothness near the exceptional locus. The centers are global and disjoint, so the blow-ups are projective. They are contained in the affine Weierstrass chart and do not meet the zero section.
\end{proof}

\section{Ordinary double points and global small resolutions}\label{sec:nodes}

We determine the singularities away from the boundary and construct their small resolutions.

\subsection{The ordinary double points}
Put
\begin{equation}\label{eq:Jdef}
 J(v)=a^3B(v)^4-b^4A(v)^3.
\end{equation}
A useful factorization, with the constants \eqref{eq:constants}, is
\begin{equation}\label{eq:Jfactor}
 \begin{split}
 J(v)={}&-(v^4-16)(16v^4-1)(v^4-81)(81v^4-1)\\
        &\hspace{7mm}\cdot(25v^4-14v^2+25)(25v^4+14v^2+25).
 \end{split}
\end{equation}
In particular, $J$ has twenty-four distinct roots, all in $\QQ(i)$, and none is a zero of $A$, $B$, or $C$.

\begin{Pro}\label{prop:nodes}
Away from the two singular curves in Lemma~\ref{lem:boundary}, the Weierstrass model $\WW$ has exactly ninety-six singular points. They are ordinary double points. More precisely, for every root $v_0$ of $J$, the four corresponding points are
\begin{equation}\label{eq:nodeformula}
 t^4=\frac{C(v_0)}{c},\qquad
 y=0,\qquad x=\frac{3a^2C(v_0)}{cA(v_0)}t^2.
\end{equation}
\end{Pro}
\begin{proof}
Work first with $A(v)B(v)t\ne0$. In \eqref{eq:nonminimal} set
\[
 q=t^4,\quad k=aA(v),\quad d=54b^4,\quad h=54B(v)^4,
 \quad p=cC(v),\quad w=p+dq+h/q.
\]
The equation is
\begin{equation}\label{eq:normalizednode}
 F=y_0^2-x_0^3+27kx_0-54w=0.
\end{equation}
At a singular point, $y_0=0$, $x_0^2=9k$, and $w_q=0$. The last equation gives
\[
 q=\varepsilon\frac{B(v)^2}{b^2},\qquad \varepsilon\in\{1,-1\}.
\]
At this critical value of $q$, one has
\[
 w_\varepsilon=cC(v)+108\varepsilon b^2B(v)^2.
\]
The identity \eqref{eq:nodekeyidentity} reads
$w_\varepsilon^2-k^3=-108L_\varepsilon(v)^2$, where
$L_\varepsilon(v)=b^2C(v)-\varepsilon cB(v)^2$.
The equation $F=0$, together with $x_0^2=9k$, is equivalent to $w=kx_0/3$. Consequently \eqref{eq:nodekeyidentity} forces
\[
 L_\varepsilon(v):=b^2C(v)-\varepsilon cB(v)^2=0.
\]
Moreover,
\begin{equation}\label{eq:JLL}
 J(v)=-L_+(v)L_-(v).
\end{equation}
At a root of $J$, exactly one of the two factors vanishes, and it vanishes simply. The resulting value of $q$ is $C(v)/c$. Substitution gives $x_0=3a^2C(v)/(cA(v))$, hence \eqref{eq:nodeformula}.

We now show that each of these points is an ordinary double point. Near such a point choose the analytic branch $\kappa(v)$ with $\kappa^2=k^3$ and $\kappa(v_0)=w_\varepsilon(v_0)$. The critical value of $F$ after solving $F_{x_0}=F_{y_0}=F_q=0$ is
\[
 -54\bigl(w_\varepsilon(v)-\kappa(v)\bigr).
\]
By \eqref{eq:nodekeyidentity},
\begin{equation}\label{eq:doublecriticalvalue}
 w_\varepsilon-\kappa
 =\frac{-108L_\varepsilon^2}{w_\varepsilon+\kappa}.
\end{equation}
The denominator is a unit and $L_\varepsilon$ has a simple zero. The critical value therefore has a zero of order exactly two in $v-v_0$. In particular its first derivative vanishes, which proves $F_v=0$.

The Hessian in the three variables $(x_0,y_0,q)$ at the critical point is diagonal, with entries
\[
 -6x_0,\qquad 2,\qquad -108d/q,
\]
all nonzero. Equation~\eqref{eq:doublecriticalvalue} shows that its Schur complement in the full Hessian is also nonzero. Hence the full Hessian is nondegenerate. Since $q=t^4$ is an \'etale change of coordinate here, each point is an ordinary double point of the original Weierstrass equation.

It remains to examine the points excluded by $AB=0$. If $B=0$ and $t\ne0,\infty$, then $w_q=d\ne0$. If $A=0$, a possible singular point would have $x_0=y_0=w=0$. But the quadratic $dq^2+pq+h$ has nonzero discriminant there, since
\[
 p^2-4dh=-108a^3B(v)^4\qquad\text{modulo }A(v).
\]
Its roots in $q$ are simple, so $F_q\ne0$. The polynomials $A$ and $B$ have no common zero. Lemma~\ref{lem:boundary} covers $t=0,\infty$ and the exceptional boundary, including the chart $v=\infty$. Thus there are no further singular points. Finally, each of the twenty-four roots of $J$ gives four distinct nonzero values of $t$, proving the count.
\end{proof}

\subsection{Section divisors through the double points}
Let $P_1,\ldots,P_6,\Pp1,\ldots,\Pp6$ be the sections specified in Appendix~\ref{app:sections}. We shall use only
\begin{equation}\label{eq:sixsections}
 P_1,\ P_2,\ P_6,\ \Pp1,\ \Pp2,\ \Pp6
\end{equation}
as centers of small resolutions.

\begin{Lem}\label{lem:cover}
Every ordinary double point of Proposition~\ref{prop:nodes} lies on the image of at least one section in \eqref{eq:sixsections}. The indicated image is a smooth local surface at the double point.
\end{Lem}
\begin{proof}
All coordinates of the sections are polynomial in $v,t$ in the affine chart containing the double points. Their images are therefore graphs over the smooth base and are smooth surfaces there. Table~\ref{tab:nodecover} gives a fourth root $\tau$ of $C(v_0)/c$ for every root $v_0$ of $J$, and a section passing through the point for each $t=i^j\tau$, $0\le j\le3$. Each table entry is checked by the two substitutions
\[
 x_P(v_0,i^j\tau)=\frac{3a^2C(v_0)}{cA(v_0)}(i^j\tau)^2,
 \qquad y_P(v_0,i^j\tau)=0.
\]
Substitution of the formulas in Appendix~\ref{app:sections} verifies these identities in $\QQ(i)$ for all ninety-six points.
\end{proof}

\begin{Lem}\label{lem:smallresolution}
Let $W$ be a threefold with an ordinary double point and a local Weierstrass equation $y^2=x^3+f(s,t)x+g(s,t)$. Suppose a regular section $(p(s,t),q(s,t))$ passes through the double point. Blowing up the ideal of its graph gives a smooth small crepant resolution near that point.
\end{Lem}
\begin{proof}
Put $U=x-p$, $V=y-q$, and $h_*=3p^2+f$. The exact equation becomes
\begin{equation}\label{eq:sectionfactor}
 V(V+2q)=U(U^2+3pU+h_*).
\end{equation}
At the double point $q=h_*=0$. The determinant of the Hessian of \eqref{eq:sectionfactor} is
\[
 4\left(\det\frac{\partial(q,h_*)}{\partial(s,t)}\right)^2.
\]
Since the point is ordinary double, $\dd q$ and $\dd h_*$ are independent base differentials.

Blow up $(U,V)$. On the chart $V=\lambda U$, the strict transform is
\begin{equation}\label{eq:smallchartone}
 U\lambda^2+2q\lambda-U^2-3pU-h_*=0.
\end{equation}
Along the exceptional fiber, its base differential is $2\lambda\dd q-\dd h_*$, which cannot vanish. On the other chart $U=\mu V$, the equation is
\begin{equation}\label{eq:smallcharttwo}
 V+2q-\mu^3V^2-3p\mu^2V-\mu h_*=0,
\end{equation}
and the corresponding differential is $2\dd q-\mu\dd h_*$, again nonzero. These charts are smooth and the exceptional fiber is a single $\PP^1$. Thus the resolution is small. A small resolution of this Gorenstein singularity is crepant, since there is no exceptional divisor on which a discrepancy could be supported.
\end{proof}

\begin{Pro}\label{prop:globalflat}
The Weierstrass model $\WW\to\BB$ admits a projective crepant resolution $X\to\WW$ such that $X\to\BB$ is flat and has a regular zero section.
\end{Pro}
\begin{proof}
The Weierstrass hypersurface is Cohen--Macaulay and its singular locus has codimension at least two, by Lemma~\ref{lem:boundary} and Proposition~\ref{prop:nodes}. It is therefore normal. First apply Lemma~\ref{lem:A2resolve}. The resulting projective normal threefold is smooth away from the ninety-six double points. Take the closures of the six section images in \eqref{eq:sixsections}. These are global effective Weil divisors. At every smooth point their ideals are invertible, so blowing them up is an isomorphism there. Near a double point on a given section, the ideal is the graph ideal $(U,V)$ from Lemma~\ref{lem:smallresolution}; hence its blow-up is a smooth small resolution.

Blow up successively the ideal sheaves of the strict transforms of these six divisors. Once a double point has been resolved, every later divisor is Cartier in its smooth neighborhood and later blow-ups are isomorphisms there. Lemma~\ref{lem:cover} shows that all double points are resolved. Each blow-up is a projective morphism defined by a global ideal sheaf, and every step is crepant.

The zero section of the Weierstrass model is smooth and is disjoint from all singular centers. The later divisorial ideals are invertible near it, even where their supports meet it. The zero section is therefore unchanged and remains regular.

Every fiber of the resolved morphism has dimension one. The curve blow-ups add curves over the two discriminant curves, and the small resolutions add only $\PP^1$'s over the isolated base points. They introduce no surface over a point of $\BB$. Since $X$ is Cohen--Macaulay and $\BB$ is smooth, the equidimensional morphism is flat by miracle flatness. Its fibers are connected: its generic fiber is geometrically connected, and Stein factorization over the normal base has a finite birational factor, which is an isomorphism.
\end{proof}

\begin{Cor}\label{cor:XCY}
The threefold $X$ of Proposition~\ref{prop:globalflat} is Calabi--Yau, and $\rk\MW(X/\BB)=12$.
\end{Cor}
\begin{proof}
Let $\nu:X\to\WW$ be the crepant resolution of
Proposition~\ref{prop:globalflat}. Equation~\eqref{eq:Lbaseblow} gives
$\omega_{\WW}\simeq\OO_{\WW}$, and hence
\[
 \omega_X\simeq\nu^*\omega_{\WW}\simeq\OO_X.
\]
Proposition~\ref{prop:genericidentification} identifies the function
fields of $X$ and $Y$ with that of the same elliptic curve over
$\CC(v,t)$. Thus $X$ and $Y$ are birational smooth projective threefolds.
Birational invariance of holomorphic forms, together with Hodge
symmetry and Proposition~\ref{prop:QCY}, gives
\[
 h^j(X,\OO_X)=h^j(Y,\OO_Y)=0,\qquad j=1,2.
\]
This proves that $X$ is Calabi--Yau. Its Mordell--Weil group is the
group of points of \eqref{eq:mainW} over $\CC(v,t)$, whose rank is
twelve by Proposition~\ref{prop:rank12}.

The two models have different morphisms to the base: $Y\to\FF_3$ has
surface fibers, whereas $X\to\BB$ is flat by
Proposition~\ref{prop:globalflat}. Flatness is not inferred from their
birational equivalence.
\end{proof}

\begin{Rem}\label{rem:constructionmethod}
One may seek further examples by changing the constant elliptic curve or
the varying family in the product construction. The associated Kummer
surfaces and a suitable base change must provide an explicit K3 model
whose coefficient degrees allow an anticanonical compactification,
as in Section~\ref{sec:origin}. The singularities of the
compactification must then admit a projective crepant resolution, and the
vanishing of $H^1(\OO)$ and $H^2(\OO)$ must be checked.
To obtain a flat elliptic model, one must also control the degenerations
over the base. In the present example, the sections used for the rank
calculation provide global divisors for the small resolutions. For another
family, their field of definition, independence, and incidence with the
singularities require separate verification.
\end{Rem}

\section{Hodge numbers}\label{sec:hodge}

We compute the Picard number and Euler characteristic using the flat elliptic fibration $X\to\BB$.

\begin{Pro}\label{prop:Picard}
One has $\rho(X)=25$.
\end{Pro}
\begin{proof}
Let $K=\CC(\BB)$. Restrict divisors on the smooth threefold $X$ to the generic elliptic curve $\EE/K$. Every closed point of $\EE$ extends by closure to a divisor on $X$, which is Cartier because $X$ is smooth. Restriction therefore surjects onto $\Pic(\EE)$. The rational zero splits its degree map, so its degree-zero part is $\EE(K)$ and
\[
 \rk\Pic(\EE)=1+\rk\EE(K)=13.
\]
The kernel is generated by vertical divisors: a rational function giving a principal divisor on $\EE$ belongs to $K(X)$, and subtracting its divisor leaves a vertical divisor.

Flatness excludes divisors over points of $\BB$. Every vertical prime divisor therefore lies over a curve. Over all curves other than the two split $\IV$ curves the generic fiber is irreducible; the zero section excludes multiple fibers. Their vertical classes are pullbacks from the base. Over each split $\IV$ curve there are exactly two further independent component classes. Their independence modulo pullbacks follows by intersecting with the nonidentity components of a transverse surface fiber: the resulting matrix is the nonsingular negative $\Aroot$ Cartan matrix. Pullback from $\BB$ is injective because it has a left inverse given by the zero section. Since $\rho(\BB)=2+6=8$, the vertical rank is $8+2+2=12$.

Finally, $H^1(X,\OO_X)=0$ by Corollary~\ref{cor:XCY}, so $\Pic^0(X)=0$. Taking ranks in the restriction sequence gives
\begin{equation}\label{eq:flatPicard}
 \rho(X)=\rho(\BB)+1+\rk\MW(X/\BB)+2+2
        =8+1+12+2+2=25.
\end{equation}
\end{proof}

Let $C_+,C_-\subset\BB$ be the strict transforms of $S_+,S_-$, and let $D$ be the residual discriminant after removing their multiplicity-four contributions.

\begin{Lem}\label{lem:residualdiscriminant}
The curve $D$ is reduced and disjoint from $C_+\cup C_-$. It has arithmetic genus $177$, exactly $96$ ordinary nodes and $64$ ordinary cusps, and no other singularities.
\end{Lem}
\begin{proof}
The blow-up description and adjunction give
\begin{equation}\label{eq:baseintersections}
 K_{\BB}^2=2,\qquad C_+^2=C_-^2=-3,\qquad
 K_{\BB}\cdot C_+=K_{\BB}\cdot C_-=1,\qquad C_+\cdot C_-=0.
\end{equation}
The discriminant is a section of $-12K_{\BB}$, and its multiplicity along each $C_\pm$ is four. Thus
\begin{equation}\label{eq:residualclass}
 D=-12K_{\BB}-4C_+-4C_-.
\end{equation}
Intersecting \eqref{eq:residualclass} gives
\begin{equation}\label{eq:residualnumbers}
 D^2=384,\qquad K_{\BB}\cdot D=-32,\qquad
 D\cdot C_+=D\cdot C_-=0,\qquad
 p_a(D)=1+\frac{384-32}{2}=177.
\end{equation}

The degree-sixteen residual discriminant over $\CC(v)$ is square-free by Lemma~\ref{lem:K3fibers}. There is no component over a point of the ruling: on the finite $v$-chart its leading coefficient in $t$ is a nonzero constant; the chart at $v=\infty$ is identical under \eqref{eq:Wtransition}. Over a base exceptional curve the discriminant of \eqref{eq:firstEchart} is not identically zero. Hence $D$ is reduced. Lemma~\ref{lem:boundary} shows that it is disjoint from $C_+\cup C_-$ and smooth at its remaining boundary points. More explicitly, in \eqref{eq:firstEchart}, a discriminant zero has $z\ne0$ and $g_E\ne0$, and
$\partial g_E/\partial z=11664\delta(0)^4z^3\ne0$; hence the discriminant is smooth there.

Consider first a discriminant point with $(f,g)\ne(0,0)$. The plane cubic has an ordinary node, with nondegenerate Hessian in its two fiber coordinates. The holomorphic Morse lemma with parameters puts the total equation locally in the form
\[
 u_1^2+u_2^2-h(s_1,s_2)=0;
\]
up to a unit its discriminant is $h=0$. The total space is singular precisely when $h=\dd h=0$. At each of the $96$ points of Proposition~\ref{prop:nodes}, the full Hessian is nondegenerate, so the Hessian of $h$ is nondegenerate. The corresponding discriminant point is therefore an ordinary node. Conversely, every singular discriminant point with $(f,g)\ne(0,0)$ gives a singularity of the total Weierstrass space and is on that complete list.

It remains to consider $f=g=0$ away from the boundary. The polynomial $A(v)$ has eight simple roots and none is a root of $B$: indeed
$A'(v)=8v^3(v^4+7)$, while $A(0)=1$ and $A(v)=-48$ if $v^4=-7$; also $A=16$ if $v^4=1$. At a zero of $A$, the equation $g=0$ with $t\ne0$ becomes
\begin{equation}\label{eq:cuspquadratic}
 d\theta^2+p\theta+h=0,\qquad
 \theta=t^4,\quad d=54b^4,\quad p=cC(v),\quad h=54B(v)^4.
\end{equation}
By \eqref{eq:ABCidentity} and $a^3-c^2=108b^4$,
\begin{equation}\label{eq:cuspdiscriminant}
 p^2-4dh\equiv-108a^3B(v)^4\not\equiv0\pmod{A(v)}.
\end{equation}
Thus the quadratic has two distinct nonzero roots, and each gives four distinct nonzero values of $t$. There are exactly $8\cdot2\cdot4=64$ such points. At each of them $\partial f/\partial v\ne0$, $\partial f/\partial t=0$, and $\partial g/\partial t\ne0$. Hence $f,g$ are local base coordinates and the discriminant has the ordinary cusp equation $4f^3+27g^2=0$. The node list and the boundary analysis exclude any other singularity of $D$.
\end{proof}

\begin{Pro}\label{prop:Euler}
One has $e(X)=48$ and $(h^{1,1}(X),h^{2,1}(X))=(25,1)$.
\end{Pro}
\begin{proof}
For any reduced projective curve with $n$ ordinary nodes and $c$ ordinary cusps, its normalization $\nu:\widetilde D\to D$ satisfies
\[
 \chi(\OO_{\widetilde D})=\chi(\OO_D)+n+c.
\]
Every node identifies two points of the normalization; a cusp does not change the number of points. Since $e(\widetilde D)=2\chi(\OO_{\widetilde D})$, this gives, without assuming irreducibility,
\[
 e(D)=2-2p_a(D)+n+2c.
\]
Lemma~\ref{lem:residualdiscriminant} therefore gives
\begin{equation}\label{eq:residualEuler}
 e(D)=2-2\cdot177+96+2\cdot64=-128,\qquad
 e(D^{\mathrm{sm}})=-128-96-64=-288.
\end{equation}

Over $D^{\mathrm{sm}}$ the fiber has type $\Ione$ and Euler characteristic one. At a cusp the fiber is a cuspidal cubic, with Euler characteristic two; the total space is smooth there. At each discriminant node the small resolution replaces the node of a nodal cubic by $\PP^1$, so the fiber has Euler characteristic $1-1+2=2$. Along each $C_\pm\simeq\PP^1$ the fiber has type $\IV$ everywhere, with Euler characteristic four, by the uniform boundary resolution and the disjointness of the residual discriminant. Smooth elliptic fibers contribute zero. Additivity of Euler characteristic over these strata gives
\begin{equation}\label{eq:flatEuler}
 e(X)=4\bigl(e(C_+)+e(C_-)\bigr)+e(D^{\mathrm{sm}})+2\cdot96+2\cdot64
      =16-288+192+128=48.
\end{equation}
Since $X$ is Calabi--Yau, Proposition~\ref{prop:Picard} gives $h^{1,1}(X)=\rho(X)=25$. Finally $e(X)=2(h^{1,1}(X)-h^{2,1}(X))$, so $h^{2,1}(X)=1$.
\end{proof}

\begin{Rem}\label{rem:quarticEuler}
The Euler characteristic can also be computed on the quartic resolution $Y$. The singularity on each special quartic fiber is analytically $z^2-q(x,y)=0$, after absorbing the unit multiplying $z^2$. Its Milnor number is
\[
 \dim_{\CC}\CC[[x,y]]/(q_x,q_y)=3\cdot3=9,
\]
since the two cubic partial derivatives form a homogeneous complete intersection. The Milnor-fiber formula \cite[Theorems~6.5 and~7.2]{Milnor} gives $e(\Qquart_{v_0})=24-9=15$. A degree-two del Pezzo exceptional surface is a double plane branched over a smooth genus-three quartic and has Euler characteristic $2\cdot3-(2-2\cdot3)=10$. Thus $e(Y_{v_0})=15-1+10=24$, and
\[
 e(Y)=(2-6)\cdot24+6\cdot24=48.
\]
\end{Rem}

\begin{proof}[Proof of Theorem~\ref{thm:main}]
We take $X$ to be the resolution of $\WW$ constructed in
Proposition~\ref{prop:globalflat}. That proposition gives projectivity,
flatness, connected fibers, and a regular zero section.
Corollary~\ref{cor:XCY} proves that $X$ is Calabi--Yau, using the
quartic resolution $Y$ only for the vanishing of $H^1(\OO)$ and
$H^2(\OO)$. Proposition~\ref{prop:genericidentification} gives the
stated generic equation, and Proposition~\ref{prop:rank12} determines
its torsion-free Mordell--Weil group of rank twelve. Lemma~\ref{lem:boundary}, Proposition~\ref{prop:nodes}, and Lemma~\ref{lem:residualdiscriminant} show that the only reducible codimension-one fibers are the two split $\IV$ fibers. The other discriminant curves have generic fiber $\Ione$; the cuspidal fibers occur only at isolated points. Propositions~\ref{prop:Picard} and \ref{prop:Euler} give the Hodge numbers directly on $X$.
\end{proof}

\section{Variation of the fixed elliptic curve}\label{sec:variation}

Varying $u$ gives resolutions $Y_u$ of the quartic model, with $Y_2=Y$.
We determine the rank over the parameter field, the deformations of $Y$,
the birational type of the resulting fourfold, and the Galois
representation on $H^3(Y_2)$.
The flat model $X$ of Theorem~\ref{thm:main} was constructed for $u=2$;
the deformation statements below concern $Y_u$.

\subsection{Rank over the parameter field and isogenous pairs}

\begin{Pro}\label{prop:threeparamrank}
The elliptic curve
\begin{equation}\label{eq:threeparam}
\begin{split}
 \EE_{u,v}:\quad y^2={}&x^3-27A(u)A(v)t^4x\\
 &+2916\bigl(B(u)^4t^{10}+B(v)^4t^2\bigr)
   +54C(u)C(v)t^6
\end{split}
\end{equation}
has Mordell--Weil rank twelve over $\QQ(i)(u,v,t)$.
\end{Pro}
\begin{proof}
The formulas of Appendix~\ref{app:sections}, evaluated at $u$ and $iu$,
give twelve points of \eqref{eq:threeparam}. Good reduction at the
valuation $u-2$ specializes them to the independent points of
Proposition~\ref{prop:rank12}, proving the lower bound twelve.
Since $u,v$ are independent parameters, $E_u$ and $E_v$ are
geometrically non-isogenous. The upper bound follows from
\cite[Proposition~2.8]{KK}, as in Proposition~\ref{prop:rank12}.
\end{proof}

\begin{Rem}\label{rem:isogenies}
For a smooth pair $E_u,E_v$ with distinct $j$-invariants, the geometric
Mordell--Weil rank of $F^{(4)}$ is
\[
 12+\rk\Hom(E_u,E_v),
\]
by \cite[Proposition~2.8]{KK}. A nontrivial modular relation
$\Phi_N(j(E_u),j(E_v))=0$ therefore gives geometric rank thirteen
generically on any component where both parameters vary and the curves
are not isomorphic. The additional section may require a finite
extension of the component's function field. Since the restriction
changes the coefficient degrees, a Calabi--Yau model requires a new
anticanonical calculation as in Section~\ref{sec:origin}.
For fixed $u=2$, an isogenous specialization of $E_v$ changes the rank
of one K3 fiber; the threefold rank over $\CC(v,t)$ remains twelve.

Equal $j$-invariants require separate treatment: reducible fibers change,
and the rank formula has a correction of four if the common $j$ is
different from $0,1728$, and of eight otherwise
\cite[Proposition~2.8]{KK}. In the present threefold,
$J(v)=0$ is exactly the condition $j(E_v)=j(E_2)$.
These are the twenty-four parameters that produce the ordinary double
points in Proposition~\ref{prop:nodes}.
\end{Rem}

\subsection{Deformations of the quartic resolution}

Put $U=\mathbb A^1\setminus\{0,1,-1,i,-i\}$ and replace the fixed
binary quartic $q$ by
\[
 q_u(X_0,X_1)=X_0X_1\bigl(X_0+(u^2+1)^2X_1\bigr)
                       \bigl(X_0+(u^2-1)^2X_1\bigr).
\]
The resulting quartic in the same bundle $M$ is denoted by $\Qquart_u$.
Its weighted resolution is denoted by $Y_u$, so $Y_2=Y$.

\begin{Pro}\label{prop:deformations}
The threefolds $Y_u$, $u\in U$, form a smooth projective family of
Calabi--Yau threefolds with $h^{2,1}(Y_u)=1$. Its Kodaira--Spencer map
at $u=2$ is an isomorphism. In particular, the family is locally complete
there, and none of its smooth fibers is rigid.
\end{Pro}
\begin{proof}
For $u\in U$, the polynomial $q_u$ has four distinct projective zeros.
The six weighted blow-up centers of Proposition~\ref{prop:Qresolution}
are independent of $u$, and its charts give a smooth projective family
$\mathcal Y\to U$. Adjunction and Proposition~\ref{prop:QCY} show
that its fibers are Calabi--Yau. By \cite[Theorem~1.1]{Batyrev}, the
smooth projective birational varieties $X$ and $Y_2$, both with trivial
canonical bundles, have equal Betti numbers. Hence $b_3(Y_2)=4$ and
$h^{2,1}(Y_2)=1$. Constancy of Hodge numbers over the connected base
$U$ gives $h^{2,1}(Y_u)=1$.

The exceptional divisor over $v=0$ is the degree-two del Pezzo surface
\[
 D_u:\quad -z^2+4s^4-q_u(x,y)=0
       \quad\text{in }\PP(1,1,1,2).
\]
Adjunction gives $N_{D_2/Y_2}\simeq K_{D_2}$ and
$H^0(D_2,N_{D_2/Y_2})=0$. A trivial first-order deformation of $Y_2$
would therefore make $D_u$ constant, since embedded deformations of
$D_2$ in $Y_2$ have tangent space $H^0(D_2,N_{D_2/Y_2})$.
The branch curve of its anticanonical double cover would also be constant:
\[
 C_u:\quad q_u(x,y)=4s^4.
\]
On this smooth plane quartic, $H^0(C_2,T_{C_2})=0$ forces the
involution $s\mapsto-s$ to be constant under a first-order
trivialization. Its quotient is $E_u$, via $\xi=x/y$ and
$\eta=2(s/y)^2$, so $E_u$ would be constant to first order. However,
\eqref{eq:jEz} gives
\[
 \frac{\dd j(E_u)}{\dd u}
   =\frac{64A(u)^2C(u)}{B(u)^5},
 \qquad \left.\frac{\dd j(E_u)}{\dd u}\right|_{u=2}\ne0.
\]
Thus the Kodaira--Spencer map is nonzero between one-dimensional
spaces, hence an isomorphism. The Bogomolov--Tian--Todorov theorem
gives unobstructedness, local completeness at $u=2$, and, since
$h^{2,1}(Y_u)=1$, non-rigidity of every smooth fiber.
Degenerations outside $U$ are not covered by this statement.
\end{proof}

\subsection{An order-four action on third cohomology}

The set of six weighted blow-up ideals on $\Qquart_2/\QQ$ is Galois-stable,
so their normalized blow-ups define $Y_2/\QQ$.
Write $G_K=\operatorname{Gal}(\overline{\QQ}/K)$.

\begin{Cor}\label{cor:H3splitting}
Let $\ell$ be a prime and let $L=\QQ_\ell(i)$. There is a decomposition
\[
 H^3_{\mathrm{\acute et}}(Y_{2,\overline{\QQ}},\QQ_\ell)
       \otimes_{\QQ_\ell}L=V_i\oplus V_{-i},
 \qquad \dim_LV_i=\dim_LV_{-i}=2,
\]
whose summands are stable under $G_{\QQ(i)}$. Elements of
$G_{\QQ}\setminus G_{\QQ(i)}$ interchange them.
\end{Cor}
\begin{proof}
The automorphism
\[
 \iota:(v;X_0,X_1,X_2,X_3)\longmapsto
             (iv;X_0,X_1,X_2,X_3)
\]
preserves $\Qquart_u$: it interchanges $\alpha(v)$ and $\beta(v)$.
On the chart at infinity it acts by
\[
 (w;X_0',X_1',X_2',X_3')\longmapsto
 (-iw;iX_0',iX_1',X_2',X_3').
\]
It permutes the six weighted blow-up ideals and therefore lifts to
$\mathcal Y$ over $\QQ(i)$. The residue form satisfies
$\iota^*\Omega_u=i\Omega_u$. Since $\iota$ acts over the identity
on $U$, its action commutes with the Gauss--Manin connection.
By Proposition~\ref{prop:deformations}, the $H^{2,1}$ component of
$\nabla_{\partial/\partial u}\Omega_u$ is nonzero at $u=2$.
Thus $\iota$ acts by $i$ on $H^{3,0}(Y_2)\oplus H^{2,1}(Y_2)$ and,
by complex conjugation, by $-i$ on $H^{1,2}(Y_2)\oplus H^{0,3}(Y_2)$.

Comparison with \'etale cohomology gives the two eigenspaces of
dimension two over $L$. Because $\iota$ is defined over $\QQ(i)$,
they are stable under $G_{\QQ(i)}$. Conjugating $i$ to $-i$ sends
$\iota$ to $\iota^{-1}$ and hence interchanges the eigenspaces.
\end{proof}

The decomposition requires restriction to $G_{\QQ(i)}$ and coefficients
in $L$; its summands are not $G_{\QQ}$-stable.
\subsection{The birational type of the fourfold}

For the polynomial $h_u$ of Section~\ref{sec:origin}, consider
\[
 \mathcal T:\quad s^4=h_u(\xi),
 \qquad \sigma(u,\xi,s)=(u,\xi,is).
\]
Here $\mathcal T$ denotes the normal projective cyclic cover constructed
below. Set $E_{\mathrm{CM}}:b^2=a^3-a$, with
$[i](a,b)=(-a,ib)$; the subscript distinguishes this fixed curve
from the family $E_z$.

\begin{Lem}\label{lem:cyclic-abelian}
The surface $\mathcal T$ is birational over $\QQ$ to an abelian surface
$\mathcal A$.  Over $\CC$, there is an isomorphism
\[
 (\mathcal A,\sigma)\simeq
 \bigl(E_{\mathrm{CM}}^2,([i],[i])\bigr).
\]
More precisely, the minimal resolution of $\mathcal T$ is the blow-up
of $\mathcal A$ at four points.
\end{Lem}
\begin{proof}
Let $S_0$ and $F$ be the negative section and a fiber of
$R=\FF_4$. The complement of $S_0$ is the total space of
$\OO_{\PP^1}(4)$, with transition $\xi'=w^4\xi$, $w=1/u$.
The graphs
\[
 D_0=\{\xi=0\},\qquad
 D_\alpha=\{\xi=-(u^2+1)^2\},\qquad
 D_\beta=\{\xi=-(u^2-1)^2\}
\]
are disjoint from $S_0$ and have class $S_0+4F$. Thus, with
$L=S_0+3F$, the reduced divisor $D=S_0+D_0+D_\alpha+D_\beta$
has class $4L$, while $L^2=2$ and $K_R=-2L$.
Taking its fourth root gives a finite flat cover $p:\mathcal T\to R$
with $p_*\OO_{\mathcal T}=\bigoplus_{j=0}^3\OO_R(-jL)$ and the stated
function field; it is connected because $h_u$ has valuation one
along $D_0$. At infinity the fourth-root coordinate is $s'=w^3s$.

The branch divisor has precisely six order-two tangencies:
$D_\alpha,D_\beta$ at $u=0,\infty$;
$D_0,D_\beta$ at $u=\pm1$; and $D_0,D_\alpha$ at $u=\pm i$.
After an analytic coordinate change and rescaling $s$ by a unit,
the covering equation is $s^4=x(x-r^2)$, or
$z^2=r^4+4s^4$ with $z=2x-r^2$.
These are its only singularities, so the hypersurface $\mathcal T$
is normal.

The weighted blow-up with weights $(1,1,2)$ on $(r,s,z)$ has
exceptional curve $E\subset\PP(1,1,2)$ given by $z^2=r^4+4s^4$.
This is a smooth genus-one double cover of $\PP^1$ branched at four
points, avoiding $[0:0:1]$. The weight-one charts dehomogenize $E$
and are smooth. In the weight-two chart, the index cover
$z=\rho^2$, $r=\rho r_1$, $s=\rho s_1$ has equation
$1=r_1^4+4s_1^4$; the residual involution changes all three signs
and acts freely. This chart is smooth as well.
The construction is algebraic and projective: before analytic
normalization, take $r$ as a parameter at the collision and $z$ as
the fiber coordinate minus the midpoint of the colliding sections,
and use the normalized blow-up of $(z)+(r,s)^2$ in the smooth
ambient threefold. Its weighted initial equation is
$\lambda z^2+\mu r^4-s^4$, with $\lambda\mu\ne0$, so the same chart
calculation applies. Since the exceptional curves are elliptic,
this is the minimal desingularization; uniqueness makes it defined
over $\QQ$.

Write $f:\widetilde{\mathcal T}\to\mathcal T$ for this resolution
and $E_1,\ldots,E_6$ for its exceptional curves.
Their normal bundles are $\OO_{E_j}(-1)$ of degree $-2$, and each
has discrepancy $1+1+2-1-4=-1$.
Let $R_0,R_\alpha,R_\beta,R_\infty$ be the strict transforms of the
ramification curves. The tautological fourth-root section has order
one along each $R_\nu$ and each $E_j$, while the cyclic-cover formula
gives $\omega_{\mathcal T}=p^*\OO_R(K_R+3L)=p^*\OO_R(L)$.
Consequently
\begin{equation}\label{eq:cyclic-canonical}
 K_{\widetilde{\mathcal T}}
 \sim(p\circ f)^*L-\sum_jE_j
 \sim R_0+R_\alpha+R_\beta+R_\infty.
\end{equation}
The $R_\nu$ are disjoint smooth rational curves: at each former
tangency they meet $E_j$ at the distinct points $s=0$, $z=\pm r^2$.
Adjunction gives $-2=(K_{\widetilde{\mathcal T}}+R_\nu)\cdot R_\nu
=2R_\nu^2$. Contracting these four $(-1)$-curves therefore gives
a smooth projective surface $\mathcal A$ with
$K_{\mathcal A}\simeq\OO_{\mathcal A}$.

Riemann--Roch on $R$ gives
\[
 \chi(\OO_{\mathcal T})
 =\sum_{j=0}^3\chi(\OO_R(-jL))
 =\sum_{j=0}^3(1+j^2-2j)=6.
\]
At each singularity, the ideal of $E_j$ and the pulled-back maximal
ideal have the same radical. Thus the thickenings $nE_j$ form a
cofinal system of infinitesimal neighborhoods of the fiber. Their successive
kernels $\OO_{E_j}(-nE_j)$ have degree $2n>0$, hence vanishing first
cohomology. Thus $H^1(\OO_{nE_j})=H^1(\OO_{E_j})=\CC$ for $n\ge1$.
Formal functions gives length one for $R^1f_*\OO_{\widetilde{\mathcal T}}$
at each singularity; there are no higher direct images. Hence
$\chi(\OO_{\mathcal A})=\chi(\OO_{\widetilde{\mathcal T}})=6-6=0$.
Together with $K_{\mathcal A}\simeq\OO_{\mathcal A}$, this gives
$p_g=1$, $q=2$, so the classification identifies $\mathcal A$ as
an abelian surface \cite[Chapter~VIII]{Beauville}.

The four contractions are defined over $\QQ$, and the image of
$R_\infty$ gives a rational origin fixed by $\sigma$, which is defined
over $\QQ(i)$. Near $R_\infty$, the coordinates $q=1/\xi$, $z=s/\xi$
satisfy $z^4=q(1+(u^2+1)^2q)(1+(u^2-1)^2q)$.
Thus $z$ is normal to $R_\infty$, with $\sigma(z)=iz$, while
$\sigma|_{R_\infty}=1$. After contraction its derivative is
$i\,\mathrm{id}$. An automorphism of an abelian variety fixing the
origin is a group automorphism, so $\sigma$ acts on
$\mathcal A(\CC)=\CC^2/\Lambda$ by scalar multiplication by $i$.
The lattice $\Lambda$ is therefore a torsion-free rank-two
$\ZZ[i]$-module. Since $\ZZ[i]$ is a principal ideal domain, a basis
gives $\CC^2/\Lambda\simeq(\CC/\ZZ[i])^2\simeq E_{\mathrm{CM}}^2$
and carries $\sigma$ to $([i],[i])$.
\end{proof}

\begin{Thm}\label{thm:fourfold-type}
Let $\mathcal V$ be the fourfold with function field defined by
\eqref{eq:threeparam}.  Over $\CC$, it is birational to
\[
 Z_4=E_{\mathrm{CM}}^4/
       \langle([i],[i],[i],[i])\rangle.
\]
Every smooth projective model of this function field has Kodaira
dimension zero.  The quotient $Z_4$ has $16$ singularities of type
$\frac14(1,1,1,1)$ and $120$ singularities of type
$\frac12(1,1,1,1)$.  The latter are terminal, and $Z_4$ has no projective
crepant resolution.
\end{Thm}
\begin{proof}
Over $\QQ(i)(u,v)$ the calculation of Section~\ref{sec:MW}
identifies \eqref{eq:threeparam} with the pencil
$h_v(\zeta)=t^4h_u(\xi)$. On $\mathcal T\times\mathcal T$,
the ratio $t=s_v/s_u$ is invariant under $(\sigma,\sigma)$.
The invariant field is generated by $u,v,\xi,\zeta,t$ with this
relation: adjoining $s_u$ recovers the degree-four extension.
Thus $\mathcal V$ is birational to
$(\mathcal A\times\mathcal A)/\langle(\sigma,\sigma)\rangle$,
and Lemma~\ref{lem:cyclic-abelian} gives $Z_4$ over $\CC$.

On $E_{\mathrm{CM}}$, multiplication by $i$ has two fixed points
and multiplication by $-1$ has four. On its fourth power there are
therefore $16$ points with stabilizer of order four and $256-16=240$
with stabilizer of order two, the latter forming $120$ orbits.
The scalar tangent actions give the stated singularity types;
the action is free elsewhere. The invariant four-form descends to
a generator of the canonical sheaf, since $i^4=1$.
The nonidentity ages are $1,2,3$ for the order-four action and $2$
for the order-two action. By the Reid--Shepherd-Barron--Tai criterion
\cite[Theorem~2.3]{MS84}, $Z_4$ is Gorenstein and canonical,
and its order-two singularities are terminal.

For a projective resolution $g:\widehat Z_4\to Z_4$, nonnegative
discrepancies make pulled-back pluricanonical forms regular.
Conversely, forms upstairs restrict to the smooth locus downstairs
and extend by normality. Since $K_{Z_4}\sim0$ and $Z_4$ is projective
and connected,
\[
 h^0(\widehat Z_4,mK_{\widehat Z_4})
   =h^0(Z_4,mK_{Z_4})=1\qquad(m\ge1).
\]
Birational invariance gives Kodaira dimension zero for every smooth
projective model.

The order-two singularity is locally $\QQ$-factorial: pullback to
the finite smooth cover and the norm make a positive multiple of
any Weil divisor Cartier. Terminality forces any crepant resolution
there to be small. If projective, a relatively ample divisor would,
after multiplication, be the pullback of its Cartier pushforward,
contradicting positivity on a contracted curve. Such a curve exists
in a positive-dimensional fiber of a nontrivial projective
resolution. Hence $Z_4$ has no projective crepant resolution.
\end{proof}

\begin{Rem}\label{rem:fourfold-caution}
The Kodaira-dimension statement concerns smooth projective models of
the function field.  The quotient $Z_4$ has trivial canonical divisor,
but its terminal singularities prevent the preceding construction
from producing a smooth Calabi--Yau fourfold by a projective crepant
resolution.  The isomorphism
$\mathcal A_{\CC}\simeq E_{\mathrm{CM}}^2$ is geometric; no
isomorphism over $\QQ$ is asserted here.
\end{Rem}
\subsection{The product correspondence}\label{sec:product-correspondence}

Let $\mathcal S/\QQ$ be the minimal elliptic surface with generic fiber
$E_v$ in \eqref{eq:Ez}. Its six fibers of type $\mathrm I_4$, the zero
section and a fiber generate a subgroup of rank twenty in
$\NS(\mathcal S_{\CC})$, so $\rho(\mathcal S)=20$. Write
\[
 \begin{aligned}
 T(\mathcal S)&=\NS(\mathcal S_{\CC})_{\QQ}^{\perp}
       \subset H^2(\mathcal S_{\CC},\QQ),\\
 T_\ell(\mathcal S)&=
 \bigl(\NS(\mathcal S_{\overline\QQ})\otimes\QQ_\ell(-1)\bigr)^{\perp}
       \subset H^2_{\mathrm{\acute et}}
                     (\mathcal S_{\overline\QQ},\QQ_\ell),
 \end{aligned}
\]
where orthogonality is for the cup-product pairings. Both spaces have
dimension two.

\begin{Thm}\label{prop:product-motive}
There is an algebraic correspondence defined over $\QQ$ which induces
an isomorphism
\begin{equation}\label{eq:product-H3}
 H^1_{\mathrm{\acute et}}(E_{2,\overline\QQ},\QQ_\ell)
       \otimes T_\ell(\mathcal S)
 \;\simeq\;
 H^3_{\mathrm{\acute et}}(Y_{2,\overline\QQ},\QQ_\ell)
\end{equation}
of $G_{\QQ}$-representations, for every prime $\ell$.
The corresponding map of rational Hodge structures is an isomorphism
as well.
\end{Thm}

\begin{proof}
Let $\beta(\xi_v,\eta_v,v)=(\xi_v,-\eta_v,v)$ be fiberwise
inversion on $\mathcal S$, and put $P=E_2\times\mathcal S$ and
$\tau=(-1,\beta)$. Since $\beta$ is anti-symplectic,
$P/\langle\tau\rangle$ has a Borcea--Voisin crepant resolution $Z$
obtained by blowing up its singular curves \cite[\S4.1]{CG16}.
The centers are defined over $\QQ$, so this gives a smooth projective
model $Z/\QQ$.

With $h_z$ as in Section~\ref{sec:origin}, the invariant
$r=\eta_v/\eta_2$ gives
\[
 \QQ(Z)=\QQ(\xi_2,\xi_v,v,r),\qquad
 h_v(\xi_v)=r^2h_2(\xi_2).
\]
On the quartic model of $Y_2$, the coordinates
$\xi_2=X_0/X_1$, $\xi_v=X_2/X_3$, $t=X_1/X_3$ satisfy
$h_v(\xi_v)=t^4h_2(\xi_2)$. Thus $r=t^2$ defines a dominant
rational map $f:Y_2\dashrightarrow Z$ of degree two. Let
$q:P\dashrightarrow Z$ be the quotient map followed by the
birational inverse of the resolution. Their graph closures define
the correspondence over $\QQ$
\[
 \Phi=f^*q_*:H^3(P,\QQ)\longrightarrow H^3(Y_2,\QQ),
\]
with pullback and pushforward understood as graph correspondences.

The product form
$\Omega_P=(\dd\xi_2/\eta_2)\wedge
(\dd v\wedge\dd\xi_v/\eta_v)$ is $\tau$-invariant and descends
to a holomorphic form $\Omega_Z$ with $q^*\Omega_Z=\Omega_P$.
The quotient trace gives $q_*\Omega_P=2\Omega_Z$. Substituting
$r=t^2$ and differentiating $h_v(\xi_v)=t^4h_2(\xi_2)$ yields
\[
 \Phi\Omega_P=
 \frac{8t\,\dd\xi_2\wedge\dd v\wedge\dd t}
      {h_v'(\xi_v)}\ne0.
\]
In particular, $\Phi$ is nonzero on $H^{3,0}(P)$.

The Hodge structure $M=H^1(E_2,\QQ)\otimes T(\mathcal S)$ is
simple. Indeed, $j(E_2)=111284641/50625$ is not an integer,
whereas a CM $j$-invariant is an algebraic integer. Thus $E_2$
has no CM and its Hodge group is $\mathrm{SL}_2$
\cite[\S1.4]{Schnell11}. The positive definite plane
$T(\mathcal S)$, of types $(2,0),(0,2)$, has Hodge group
$\mathrm{SO}(T(\mathcal S))$. The Hodge group of the direct sum
projects onto both factors; Goursat's lemma makes it their product,
since a semisimple group and a torus have no nontrivial common
connected quotient. Both standard rational representations are
irreducible, and the first is absolutely irreducible. Their external
tensor product is therefore irreducible, proving simplicity by
\cite[Proposition~1]{Schnell11}.

Since $M$ contains $H^{3,0}(P)$, the restriction $\Phi|_M$ is
nonzero and hence injective. Both $M$ and $H^3(Y_2,\QQ)$ have
dimension four, so it is an isomorphism. The graph correspondences
are defined over $\QQ$; compatibility of cycle classes with
Betti--\'etale comparison therefore gives \eqref{eq:product-H3}
as an isomorphism of $G_{\QQ}$-representations themselves,
without taking semisimplifications.
\end{proof}

\begin{Rem}\label{rem:product-not-birational}
The K\"unneth decomposition gives
$H^3(E_2\times\mathcal S)=H^1(E_2)\otimes H^2(\mathcal S)$,
of dimension forty-four. Theorem~\ref{prop:product-motive}
identifies $H^3(Y_2)$ with its four-dimensional summand
$H^1(E_2)\otimes T(\mathcal S)$; the other forty dimensions
come from $H^1(E_2)\otimes\NS(\mathcal S)_{\QQ}$.
Both maps defining the correspondence have degree two.
\end{Rem}

\subsection{The CM characters}\label{sec:CM-characters}

Put $K=\QQ(i)$ and $L=\QQ_\ell(i)$. For the curve
$E_{\mathrm{CM}}$ and automorphism $[i]$ defined above, let
$\psi_\ell$ and $\overline\psi_\ell$ be the characters of $G_K$
on the $i$- and $-i$-eigenlines of $[i]^*$ in
$H^1_{\mathrm{\acute et}}(E_{\mathrm{CM},\overline\QQ},
\QQ_\ell)\otimes L$, respectively. These are characters on cohomology.

\begin{Cor}\label{cor:CM-description}
With compatible choices of $i$ in the coefficient and ground fields,
the summands of Corollary~\ref{cor:H3splitting} satisfy
\begin{align}
 V_i&\simeq
 \left.H^1_{\mathrm{\acute et}}(E_{2,\overline\QQ},\QQ_\ell)
             \right|_{G_K}\otimes_{\QQ_\ell}L\otimes_L\psi_\ell^2,
 \label{eq:CM-Vi}\\
 V_{-i}&\simeq
 \left.H^1_{\mathrm{\acute et}}(E_{2,\overline\QQ},\QQ_\ell)
             \right|_{G_K}\otimes_{\QQ_\ell}L
                            \otimes_L\overline\psi_\ell^{\,2}.
 \label{eq:CM-Vminus}
\end{align}
The two characters on $T_\ell(\mathcal S)|_{G_K}\otimes L$ are
$\psi_\ell^2$ and $\overline\psi_\ell^{\,2}$.
The representation $T_\ell(\mathcal S)$ is the representation of the
weight-three newform
\begin{equation}\label{eq:CM-newform}
 g(\tau)=\eta(4\tau)^6
       =q-6q^5+9q^9+10q^{13}-30q^{17}+\cdots,
 \qquad q=e^{2\pi i\tau},
\end{equation}
of level sixteen and character $\chi_{-4}$. Here $\eta$ denotes the
Dedekind eta function, and $\chi_{-4}$ is the quadratic character
associated with $\QQ(i)/\QQ$.
\end{Cor}

\begin{proof}
The change of variables over $\QQ$
\[
 \mathsf x=\frac{\xi+(v^2+1)^2}{4v^2},\qquad
 \mathsf y=\frac{\eta}{8v^3}
\]
gives Shioda's level-four model
\cite[\S4, equation~(4.1)]{Shioda73}:
\[
 \mathsf y^2=\mathsf x(\mathsf x-1)(\mathsf x-\lambda),\qquad
 \lambda=\frac14(v+v^{-1})^2.
\]
For this model, \cite[\S5, equations~(5.2)--(5.3)]{Shioda73}
gives the transcendental Frobenius polynomial
\[
 (1-\pi^2T)(1-\overline\pi^{2}T),\qquad
 p=\pi\overline\pi,\quad \pi\equiv1\pmod{2i}
\]
when $p\equiv1\pmod4$, and $(1-p^2T)^2$ over
$\mathbb F_{p^2}$ when $p\equiv3\pmod4$. These agree with the
squares of the CM characters of $E_{\mathrm{CM}}$.
The automorphism
$\rho:(v,\xi,\eta)\mapsto(iv,\xi,\eta)$ is defined over $K$
and acts by $i$ on $\dd v\wedge\dd\xi/\eta$.
Its two eigenlines in $T_\ell(\mathcal S)\otimes L$ are therefore
$G_K$-stable. Both representations over $G_K$ split into characters,
so the formulas at all good primes and Chebotarev identify their
unordered pairs of characters.

The quotient, its resolution and $r=t^2$ are equivariant for
$1\times\rho$ on $P$ and $\iota$ on $Y_2$. Thus
\eqref{eq:product-H3}, restricted to $G_K$ and extended to $L$,
carries the factors involving $T_\ell(\mathcal S)_{\pm i}$
onto $V_{\pm i}$.
The $i$-eigenlines of $\rho^*$ on $T(\mathcal S)\otimes\CC$
and of $[i]^*$ on $H^1(E_{\mathrm{CM}},\CC)$ are respectively
$H^{2,0}$ and $H^{1,0}$. Their labeled de Rham filtrations
therefore identify the character on $T_\ell(\mathcal S)_i$
as $\psi_\ell^2$, and that on $T_\ell(\mathcal S)_{-i}$
as $\overline\psi_\ell^{\,2}$. This proves
\eqref{eq:CM-Vi}--\eqref{eq:CM-Vminus} with the prescribed labels.

The square of the CM Hecke character is given on ideals prime to
$(2)$ by
\[
 \theta((\alpha))=\alpha^2,\qquad \alpha\equiv1\pmod2.
\]
Such a generator exists and is unique up to sign. This character has
infinity type two and conductor $(2)$: the conductor cannot divide
$(1+i)$, since $i\equiv1\pmod{1+i}$ but $i^2\ne1$.
Its theta series is $\eta(4\tau)^6$ \cite[\S10.1]{KSV11}.
Both its $G_{\QQ}$-representation and $T_\ell(\mathcal S)$ are
induced from either of these two distinct $G_K$-characters, so they
are isomorphic over $\QQ$. The rational change of variables above
fixes the arithmetic model, leaving no undetermined quadratic twist
in \eqref{eq:CM-newform}.
\end{proof}

\begin{Rem}
If $f_{E_2}$ is the weight-two newform attached to $E_2$, then
Theorem~\ref{prop:product-motive} identifies the good Euler factors
of $H^3(Y_2)$ with those of $f_{E_2}\otimes g$. Over $K$, its two
summands are the twists of the representation of $E_2$ in
\eqref{eq:CM-Vi}--\eqref{eq:CM-Vminus}.
\end{Rem}

\Needspace{10\baselineskip}
\appendix
\section{Explicit sections and the height pairing}\label{app:sections}

We give the sections from \cite[Sections~5.2--5.3]{KK} and compute their height pairing.

\subsection{The rational surface producing the sections}
Let $u,v$ be independent parameters. Put $\gamma=B(u)$,
$\delta=B(v)$, and use the
coordinates $x=t^2x_0$, $y=t^3y_0$. The equation is
\[
 y_0^2=x_0^3-27A(u)A(v)x_0+
 2916(\gamma^4t^4+\delta^4t^{-4})+54C(u)C(v).
\]
Under the involution $t\mapsto\delta/(\gamma t)$, the expression
$z=\gamma t+\delta/t$ is invariant. With $m=\gamma\delta$ one has
\[
 \gamma^4t^4+\delta^4t^{-4}=z^4-4mz^2+2m^2.
\]
The quotient equation is therefore the rational elliptic surface
\begin{equation}\label{eq:sectionquotient}
 y_0^2=x_0^3-27A(u)A(v)x_0+
 2916(z^4-4mz^2+2m^2)+54C(u)C(v).
\end{equation}
This is the auxiliary surface in \cite[Section~5.2]{KK}, in the present normalization.
The ansatz $x_0=a_jz+b_j$, $y_0=54z^2+c_jz+d_j$ turns the equation
into coefficient identities. The coefficients of $z^3$ and $z^2$
give, respectively,
\[
 108c_j=a_j^3,\qquad
 c_j^2+108d_j=3a_j^2b_j-11664m.
\]
The choices of $a_j,b_j$ displayed below satisfy the remaining coefficient identities.

Replacing $u$ by $iu$ preserves $A(u),C(u),B(u)^4$ but changes
$B(u)$ to $iB(u)$. It therefore gives a second set of six sections on
the same Weierstrass equation. We specialize these two sets at $u=2$.
Their independence follows from the height matrix computed below.

\subsection{The twelve sections}
Keep $u$ as a parameter and write
\[
 \gamma=B(u),\qquad \delta=B(v),\qquad m=\gamma\delta,
 \qquad z=\gamma t+\delta/t.
\]
Define six pairs $(a_j,b_j)$ by
\begin{align}
 a_1&=-18(1+i)(uv+i),& a_2&=18(1+i)(uv+i),\notag\\
 a_3&=18(1-i)(uv-i),& a_4&=18(1-i)(u+iv),\label{eq:sectiona}\\
 a_5&=-18(1+i)(u-iv),&a_6&=0,\notag
\end{align}
and
\begin{equation}\label{eq:sectionb}
\begin{aligned}
 b_1=b_2={}&-3\bigl(5u^4v^4-u^4-v^4+24iu^3v^3
                    -24u^2v^2-24iuv+5\bigr),\\
 b_3={}&-3\bigl(5u^4v^4-u^4-v^4-24iu^3v^3
                    -24u^2v^2+24iuv+5\bigr),\\
 b_4={}&3\bigl(u^4v^4-5u^4-5v^4-24iu^3v
                    +24u^2v^2+24iuv^3+1\bigr),\\
 b_5={}&3\bigl(u^4v^4-5u^4-5v^4+24iu^3v
                    +24u^2v^2-24iuv^3+1\bigr),\\
 b_6={}&3\bigl(u^4v^4+u^4+v^4+6u^4v^2+6u^2v^4
                    -12u^2v^2+6u^2+6v^2+1\bigr).
\end{aligned}
\end{equation}
Put
\begin{equation}\label{eq:sectioncd}
 c_j=\frac{a_j^3}{108},\qquad
 d_j=\frac{3a_j^2b_j-11664m-c_j^2}{108}.
\end{equation}
The relevant polynomial identity is
\begin{equation}\label{eq:rationalpointidentity}
 \begin{split}
 (54z^2+c_jz+d_j)^2={}&(a_jz+b_j)^3
     -27A(u)A(v)(a_jz+b_j)\\
 &+2916(z^4-4mz^2+2m^2)+54C(u)C(v).
 \end{split}
\end{equation}
In particular the coefficient of $z^4$ is $2916=54^2$ in this normalization. The identity follows by expansion for each of the six pairs $(a_j,b_j)$.

Define the point
\begin{equation}\label{eq:sectionformula}
 P_j(u)=\bigl(t^2(a_jz+b_j),\ t^3(54z^2+c_jz+d_j)\bigr).
\end{equation}
Since
\[
 z^4-4mz^2+2m^2=\gamma^4t^4+\delta^4/t^4,
\]
substitution in \eqref{eq:rationalpointidentity} proves that $P_j(u)$
is a point of \eqref{eq:threeparam}. Since $A(iu)=A(u)$,
$C(iu)=C(u)$, and $B(iu)^4=B(u)^4$, replacing $u$ by $iu$ gives
six further points of the same equation. Their specializations to
\eqref{eq:mainW} are denoted by
\[
 P_j=P_j(2),\qquad \Pp j=P_j(2i),\qquad 1\le j\le6.
\]
All coordinates in \eqref{eq:sectionformula} are polynomials in $t$ with coefficients in $\QQ(i)[v]$.

Two particularly simple points are
\begin{equation}\label{eq:easysections}
\begin{aligned}
 P_6={}&\bigl(3(41v^4+54v^2+41)t^2,
                 \ 54t(B(v)^2+900t^4)\bigr),\\
 \Pp6={}&\bigl(3(-7v^4+150v^2-7)t^2,
                 \ 54t(B(v)^2-900t^4)\bigr).
\end{aligned}
\end{equation}

\subsection{The height matrix}
At $v=5$ the polynomials $A$, $B$, and $J$ in \eqref{eq:Jdef} are nonzero. The calculation in Lemma~\ref{lem:K3fibers} therefore applies to this specialization: the minimal surface is a K3 surface with two split $\IV$ fibers and sixteen distinct $\Ione$ fibers. None of the twelve sections meets the zero section. The component choices can be read directly from the point formulas, before specializing $v$. For all twelve points,
\begin{equation}\label{eq:limitszero}
 \lim_{t\to0}\frac{y_P(v,t)}{t}=54B(v)^2.
\end{equation}
At infinity put $r=1/t$ and $y_{P,\infty}=r^6y_P(v,1/r)$. Since $B(2)=b$ and $B(2i)=ib$,
\begin{equation}\label{eq:limitsinfinity}
 \lim_{r\to0}\frac{y_{P,\infty}(v,r)}{r}
 =\begin{cases}
  54b^2,&P=P_j,\\
  -54b^2,&P=\Pp j.
 \end{cases}
\end{equation}
Indeed, the lowest and highest powers in the $y$-coordinate of \eqref{eq:sectionformula} are $54B(v)^2t$ and $54B(u)^2t^5$, respectively. The $t$-chart of the resolution in Lemma~\ref{lem:A2resolve} labels its two branches by these nonzero leading values. Hence all twelve sections meet the same nonidentity component at zero; at infinity the unprimed and primed points meet the two different nonidentity components.

\Needspace{16\baselineskip}
For the specialization $v=5$, in the ordered set
\[
 P_1,\ldots,P_6,\Pp1,\ldots,\Pp6
\]
the height matrix is
\begin{equation}\label{eq:heightmatrix}
 H=\frac13\begin{pmatrix}M&N\\N^{\mathsf T}&M\end{pmatrix},
\end{equation}
where
\begin{equation}\label{eq:heightblocks}
 M=\begin{pmatrix}
 8&-4&2&2&2&2\\
 -4&8&2&2&2&2\\
 2&2&8&2&2&2\\
 2&2&2&8&2&2\\
 2&2&2&2&8&2\\
 2&2&2&2&2&8
 \end{pmatrix},\qquad
 N=\begin{pmatrix}
 3&-3&-3&0&0&0\\
 -3&3&3&0&0&0\\
 3&-3&3&0&0&0\\
 0&0&0&3&3&0\\
 0&0&0&-3&3&0\\
 0&0&0&0&0&0
 \end{pmatrix}.
\end{equation}
We compute the intersection numbers as follows. At $t_0\ne0,\infty$, the local intersection multiplicity of two sections is $\min\{\ord_{t_0}(x_i-x_j),\ord_{t_0}(y_i-y_j)\}$. Their total intersection away from $0$ and $\infty$ is therefore the degree of the greatest common divisor of the coordinate differences after removal of the factor supported at $0$. In the $t$-chart of the $\Aroot$ resolution at zero, $x/t$ is a local coordinate on the chosen branch. Thus the additional intersection multiplicity is $\ord_0(x_i-x_j)-1$. The analogous expression at infinity applies only when the two sections meet the same component, after the change $r=t^{-1}$ and $x_\infty=r^4x$. Sections on different components have no intersection there.

For these twelve points the only intersection at either boundary is $P_6\cdot P'_6=1$ at $t=0$. Within each group of six, the only intersections away from the boundary are $P_1\cdot P_2=2$ and $P'_1\cdot P'_2=2$, respectively. The cross-group total intersection matrix is $\mathbf 1_{6\times6}-N/3$, where $N$ is the displayed matrix; its $(6,6)$ entry comes from the boundary. These entries are obtained by the greatest-common-divisor computation just described, with all coefficients in $\QQ(i)$.

Shioda's height formula \cite[Theorem~8.6]{Shioda} then reads
\[
 \langle P_i,P_j\rangle=2-P_i\cdot P_j
                  -\operatorname{contr}_0(P_i,P_j)
                  -\operatorname{contr}_\infty(P_i,P_j)
\]
for distinct sections. On the same nonidentity component of a split $\IV$ fiber the local contribution is $2/3$, and on different nonidentity components it is $1/3$. The self-height is $4-2/3-2/3=8/3$. These rules and the point formulas give \eqref{eq:heightmatrix}.

The leading principal minors of $H$ are
\begin{equation}\label{eq:principalminors}
 \frac83,\ \frac{16}3,\ \frac{32}3,\ \frac{64}3,
 \ \frac{128}3,\ \frac{256}3,\ \frac{1280}9,
 \ \frac{2048}9,\ \frac{2048}9,\ \frac{2048}9,
 \ \frac{2048}9,\ \frac{4096}9.
\end{equation}
In particular, $H$ is positive definite and
\[
 \det H=4096/9.
\]
This proves independence of the twelve specialized sections. For comparison, the full geometric Mordell--Weil lattice in the non-isogenous
$F^{(4)}$ case has determinant $256/9$ by \cite[Theorem~5.1]{KK}.
The twelve points above are used to prove independence; we do not assert
that they form an integral basis of $\EE(\CC(v,t))$.

\section{Sections through the ordinary double points}\label{app:nodes}

Table~\ref{tab:nodecover} lists all roots $v_0$ of \eqref{eq:Jfactor}. In every row, $\tau^4=C(v_0)/c$, and the four entries specify a divisor through the points with $t=\tau,i\tau,-\tau,-i\tau$, respectively. The coordinates $x,y$ are then fixed by \eqref{eq:nodeformula}. The values $v_0,\tau$ and the corresponding node coordinates lie in $\QQ(i)$.

For example, when $v_0=2$ the four nodes satisfy $t^4=1$ and $x=1443t^2$; the formula for $\Pp6$ gives $x=1443t^2$ and $y=0$ at all four. Every other entry is verified by the same two substitutions in \eqref{eq:sectionformula}.

The ancillary computations verify the section identities,
the height matrices at $v=5,7$, and every entry of
Table~\ref{tab:nodecover} using exact arithmetic over $\QQ(i)$.
They also verify that the section-graph Jacobian in
Lemma~\ref{lem:smallresolution} is nonzero at each of these points.
The verification code, exact node-table data and reproduction instructions
are publicly available at \url{https://github.com/Burii442/MW-rank12}.

\Needspace{35\baselineskip}
\begingroup
\small
\setlength{\tabcolsep}{10pt}
\renewcommand{\arraystretch}{1.15}
\begin{longtable}{cccccc}
\caption{A section through each of the ninety-six double points.}\label{tab:nodecover}\\
\toprule
$v_0$ & $\tau$ & $t=\tau$ & $t=i\tau$ & $t=-\tau$ & $t=-i\tau$\\
\midrule
\endfirsthead
\multicolumn{6}{c}{\tablename\ \thetable\ (continued)}\\
\toprule
$v_0$ & $\tau$ & $t=\tau$ & $t=i\tau$ & $t=-\tau$ & $t=-i\tau$\\
\midrule
\endhead
\midrule
\multicolumn{6}{r}{Continued on the next page}\\
\endfoot
\bottomrule
\endlastfoot
$-3$ & $-2 - 2 i$ & $P_{1}$ & $P_{1}$ & $P_{2}$ & $P_{2}$ \\
$-2$ & $-1$ & $\Pp{6}$ & $\Pp{6}$ & $\Pp{6}$ & $\Pp{6}$ \\
$\frac{-4 - 3 i}{5}$ & $\frac{-26 + 18 i}{125}$ & $P_{2}$ & $P_{2}$ & $P_{1}$ & $P_{1}$ \\
$\frac{-4 + 3 i}{5}$ & $\frac{-26 - 18 i}{125}$ & $P_{1}$ & $P_{2}$ & $P_{2}$ & $P_{1}$ \\
$\frac{-3 - 4 i}{5}$ & $\frac{-26 - 18 i}{125}$ & $\Pp{2}$ & $P_{1}$ & $\Pp{1}$ & $P_{1}$ \\
$\frac{-3 + 4 i}{5}$ & $\frac{-26 + 18 i}{125}$ & $\Pp{2}$ & $\Pp{1}$ & $\Pp{1}$ & $\Pp{2}$ \\
$\frac{-1}{2}$ & $\frac{-1}{8}$ & $\Pp{1}$ & $P_{1}$ & $\Pp{1}$ & $P_{1}$ \\
$\frac{-1}{3}$ & $\frac{-2 - 2 i}{27}$ & $P_{1}$ & $P_{2}$ & $P_{2}$ & $P_{1}$ \\
$- 3 i$ & $-2 - 2 i$ & $P_{1}$ & $\Pp{2}$ & $P_{1}$ & $\Pp{1}$ \\
$- 2 i$ & $-1$ & $P_{6}$ & $P_{6}$ & $P_{6}$ & $P_{6}$ \\
$\frac{- i}{2}$ & $\frac{-1}{8}$ & $P_{1}$ & $P_{1}$ & $P_{1}$ & $P_{1}$ \\
$\frac{- i}{3}$ & $\frac{-2 - 2 i}{27}$ & $\Pp{1}$ & $\Pp{2}$ & $\Pp{2}$ & $\Pp{1}$ \\
$\frac{i}{3}$ & $\frac{-2 - 2 i}{27}$ & $P_{1}$ & $\Pp{2}$ & $P_{1}$ & $\Pp{1}$ \\
$\frac{i}{2}$ & $\frac{-1}{8}$ & $P_{6}$ & $P_{6}$ & $P_{6}$ & $P_{6}$ \\
$2 i$ & $-1$ & $P_{6}$ & $P_{6}$ & $P_{6}$ & $P_{6}$ \\
$3 i$ & $-2 - 2 i$ & $\Pp{1}$ & $\Pp{2}$ & $\Pp{2}$ & $\Pp{1}$ \\
$\frac{1}{3}$ & $\frac{-2 - 2 i}{27}$ & $P_{1}$ & $P_{1}$ & $P_{2}$ & $P_{2}$ \\
$\frac{1}{2}$ & $\frac{-1}{8}$ & $P_{1}$ & $\Pp{6}$ & $P_{1}$ & $\Pp{6}$ \\
$\frac{3 - 4 i}{5}$ & $\frac{-26 + 18 i}{125}$ & $P_{1}$ & $\Pp{1}$ & $P_{1}$ & $\Pp{2}$ \\
$\frac{3 + 4 i}{5}$ & $\frac{-26 - 18 i}{125}$ & $\Pp{2}$ & $\Pp{2}$ & $\Pp{1}$ & $\Pp{1}$ \\
$\frac{4 - 3 i}{5}$ & $\frac{-26 - 18 i}{125}$ & $P_{2}$ & $P_{2}$ & $P_{1}$ & $P_{1}$ \\
$\frac{4 + 3 i}{5}$ & $\frac{-26 + 18 i}{125}$ & $P_{2}$ & $P_{1}$ & $P_{1}$ & $P_{2}$ \\
$2$ & $-1$ & $\Pp{6}$ & $\Pp{6}$ & $\Pp{6}$ & $\Pp{6}$ \\
$3$ & $-2 - 2 i$ & $P_{1}$ & $P_{2}$ & $P_{2}$ & $P_{1}$ \\
\end{longtable}
\endgroup

\providecommand{\bysame}{\leavevmode\hbox to3em{\hrulefill}\thinspace}
\providecommand{\MR}{\relax\ifhmode\unskip\space\fi MR }
\providecommand{\MRhref}[2]{%
  \href{http://www.ams.org/mathscinet-getitem?mr=#1}{#2}
}
\providecommand{\href}[2]{#2}

\end{document}